\documentclass[12pt]{article}
\usepackage{amssymb,latexsym,amsthm,todonotes}
\usepackage{amsfonts,hyperref}
\usepackage{amsmath,graphicx,psfrag}
\usepackage{makeidx,proof,url,bm}
\author{Juliette Kennedy\\
Department of Mathematics and Statistics\\ University of Helsinki,
Finland \and Jouko V\"a\"an\"anen\\
Department of Mathematics and Statistics\\ University of Helsinki,
Finland\\
ILLC, University of Amsterdam\\
Amsterdam, Netherlands }
\title{Logicality and Model Classes\footnote{The first  author would like to thank  the Academy of Finland, grant no: 322488. The second  author would like to thank  the Academy of Finland, grant no: 322795. This project has received funding from the European Research Council (ERC) under the European Union’s Horizon 2020 research and innovation programme (grant agreement No 101020762).The authors are grateful to Denis Bonnay, Dag Westerst\aa hl  and two anonymous referees for comments.}}
\begin{document}

\maketitle
\def\Mod{\mbox{Mod}}
\def\LS{\mbox{\rm LS}}
\def\Str{\mbox{\rm Str}}
\def\sp{\mbox{\rm sp}}
\def\sm{\setminus}
\def\sd{\bigtriangleup}
\def\al{\alpha}
\def\de{\delta}\def\d{\delta}
\def\be{\beta}
\def\om{\omega}
\def\ga{\gamma}
\def\fa{\forall}
\def\fo{\forall}
\def\ex{\exists}
\def\ogy{{\omega_1}}
\def\aly{{\aleph_1}}
\def\og{{\omega}}
\def\restriction{\upharpoonright}
\def\rest{\upharpoonright}
\def\Part{{\rm Part}}\def\Aut{{\rm Aut}}
\def\rank{\mathop{\rm rank}}
\def\rnk{{\rm rk}}
\def\Supp{\mathop{\rm Supp}}
\def\cof{\mathop{\rm cf}}
\def\cf{\mathop{\rm cf}}
\def\supp{\mathop{\rm Supp}}
\def\nst{{\rm NS}}
\def\suc{{\rm suc}}
\def\lim{{\rm lim}}
\def\club{{\rm CUB}}
\def\NS{{\rm NS}}
\def\CUP{{\rm CUB}}
\def\rng{\mathop{\rm rng}}
\def\len{\mathop{\rm len}}
\def\card{\mathop{\rm Card}}
\def\dom{\mathop{\rm dom}}
\def\ran{\rng}
\def\fe{\forall\exists}
\def\EF{Ehrenfeucht-Fra\"\i ss\'e\ }
\def\ef{{\rm EF}}
\def\nef{{\rm N-EF}}
\def\clubg{G}
\def\lih#1{\mathbf{#1}}
\def\hht{{\rm ht}}
\def\bdelta{\lih{\char'016}}
\def\brho{\lih{\char'032}}
\def\bx{\lih{x}}
\def\bz{\lih{z}}
\def\by{\lih{y}}
\def\ba{\lih{a}}
\def\bb{\lih{b}}
\def\bc{\lih{c}}
\def\btau{\lih{\tau}}
\def\brho{\lih{\rho}}
\def\P{{\cal P}}
\def\iG{\rotatebox[origin=c]{180}{G}}
\def\T{Tarski\ }
\def\Ts{Tarski's\ }
\def\On{{\rm On}}
\def\la{\langle}
\def\ra{\rangle}
\def\sub{\subseteq}
\def\Q{{\mathbb Q}}
\def\R{{\mathbb R}}
\def\ma{{\mathcal A}}
\def\mb{{\mathcal B}}
\def\mc{{\mathcal C}}
\def\mm{{\mathcal M}}
\def\mn{{\mathcal N}}
\def\mg{{\mathcal G}}
\def\mh{{\mathcal H}}
\def\ms{{\mathcal S}}
\def\azo{(\ma_0,\ma_1)}
\def\Str{{\rm Str}}\def\str{{\rm Str}}
\def\Seq{{\rm Seq}}
\def\kg{G\"odel\ }
\def\kgs{G\"odel's\ }
\def\phl{philosophical\ }
\def\phy{philosophy\ }
\def\phly{philosophically\ }
\def\prv{provability\ }
\def\dy{definability\ }
\def\ma{mathematics\ }
\def\mal{mathematical\ }
\def\malp{mathematical.\ }
\def\maly{mathematically\ }
\def\man{mathematician\ }
\def\mansa{mathematician's\ }
\def\mans{mathematicians\ }
\def\ff{formalism freeness\ }
\def\const{constructibility\ }
\def\nl{natural language\ }
\def\ind{independence\ }
\def\fss{formal systems\ }
\def\a{\alpha}
\def\k{\kappa}
\def\HOD{\mbox{HOD}}
\def\OD{\mbox{OD}}
\def\ZFC{\mbox{\rm ZFC}}
\def\Val{\mbox{Val}}
\def\V{\mbox{V}}
\def\DF{\mbox{D}_F}
\def\SFD{\mbox{SF}_D}
\def\BD{\mbox{B}_D}
\def\ID{\mbox{I}_D}

\def\K{\mathcal{K}}
\def\C{\mathcal{C}}
\def\mak{\mathfrak{A}}
\def\mb{\mathfrak{B}}
\def\Lo{\mathcal{L}}

\def\HOD{\mbox{HOD}}
\def\OD{\mbox{OD}}
\newcommand{\open}{\Bbb}
\newcommand{\oA}{{\open A}}
\newcommand{\oC}{{\open C}}
\newcommand{\oF}{{\open F}}
\newcommand{\oN}{{\open N}}
\newcommand{\oP}{{\open P}}
\newcommand{\oQ}{{\open Q}}
\newcommand{\oR}{{\open R}}
\newcommand{\oZ}{{\open Z}}
\def\paragr{\S}
\def\ZFC{\mbox{\rm ZFC}}
\def\ma{mathematics\ }
\def\ff{formalism freeness\ }
\def\man{mathematician\ }
\def\mans{mathematicians\ }
\def\kg{G\"odel\ }
\def\kgs{G\"odel's\ }
\def\mal{mathematical\ }
\def\maly{mathematically\ }
\def\phl{philosophical\ }
\def\phy{philosophy\ }
\def\phly{philosophically\ }
\def\phrs{philosophers\ }
\def\p{Platonism\ }
\def\d{definability\ } 
\def\pr{provability\ }
\def\a{\alpha}
\def\b{\beta}

\newtheorem{theorem}{Theorem} \newtheorem{lemma}[theorem]{Lemma}
\newtheorem{remark}[theorem]{Remark}
\newtheorem{example}[theorem]{Example}
\newtheorem{proposition}[theorem]{Proposition}
\newtheorem{observation}[theorem]{Observation}
\newtheorem{corollary}[theorem]{Corollary}
\newtheorem{definition}[theorem]{Definition}
\def\la{\langle}
\def\ra{\rangle}
\def\looo{L_{\omega_1\omega}}
\def\loo{L_{\omega\omega}}
\def\lio{L_{\infty\omega}}
\def\Loo{L_{\omega\omega}}
\def\here{{\Huge HERE.}}
\def\o{\omega}
\def\k{\kappa}
\def\b{\beta}
\def\f{\formalization}
%\tableofcontents

\begin{abstract}
We ask, when is a property of a model a logical property? According to the so-called Tarski-Sher criterion this is the case when the property is preserved by isomorphisms. We relate this to model-theoretic characteristics of abstract logics in which the model class is definable.  This results in a graded concept of logicality in the terminology of Sagi \cite{MR3794871}. We investigate which characteristics of logics, such as variants of the L\"owenheim-Skolem Theorem, Completeness Theorem, and absoluteness, are relevant from the logicality point of view, continuing earlier work by Bonnay, Feferman, and Sagi. We suggest that a logic is the more logical the closer it is to first order logic. We also offer a refinement of the result of McGee  that logical properties of models can be expressed in $L_{\infty\infty}$ if the expression is allowed to depend on the cardinality of the model, based on replacing $L_{\infty\infty}$ by a ``tamer" logic. 
\end{abstract}

\section{Introduction}

To say that  the syllogism

\begin{center}
All $A's$ are $B's$.

Some $C's$ are $A's$.

Some $C's$ are $B's$.
\end{center}

\noindent is a valid argument form is to say that the conclusion follows from the premises no matter which  predicates are substituted for the nonlogical terms $A$, $B$, $C$, whether `men', `mortal', `Greek' or `tulips', `bridges' and `pious', or whatnot, as long as the substitution is done uniformly. That the conclusion follows from the premises is a matter of {\em logic}---``is obviously true in a purely logical way," to quote Carnap.\footnote{\cite{MR0231696}, Engl. translation in \cite{logsyntax}.}

But what is it that makes a concept distinctively logical, as opposed to, say, mathematical?  How to circumscribe the logical? The question  is an urgent one  for the logician, as the  model-theoretic notion of  consequence is parasitic on the distinction  between  logical and nonlogical expressions. For on this account of logical consequence, a sentence $\phi$ is held to be a semantic consequence of a sentence $\psi$, if for every uniform substitution of the nonlogical expressions in  $\phi$  and  $\psi$, if  $\psi$  is true, then so is  $\phi$.

In 1968   in  a  (posthumously published) lecture called ``What are logical notions?"  \cite{MR868748}
 Tarski proposed  a definition of   ``logical notion," or alternatively of ``logical constant,"  modelled on the  Erlanger Program due to Felix Klein.  
%Designed for classifying geometrical notions, here  one declares  the ``notions" of metric geometry, descriptive geometry and  topological geometry, to be those invariant under the respective transformations: similarity, affine, and continuous. Thus a topological notion, for example, will be one   invariant under continuous transformations of the underlying topological space.
The core observation is the following:  for a given subject area, the number of concepts  classified as invariant are inversely related to the number of transformations---the more  transformations there are, the fewer invariant notions there are. If one  thinks  of logic  as the most general of all the \mal sciences, why not then declare  ``logical" notions to be the limiting cases? Thus a notion is to be thought of as 
logical if it is invariant under {\em all} permutations of the relevant domain. 

As expected,  the standard logical constants, namely conjunction, disjunction, negation and quantification, together with the equality relation, are all judged to be logical operations under this criterion, being all isomorphism invariant. This has the consequence that first order definability is classified as logical, {\em tout court}, being generated by these constants.\footnote{We argue for this below. V. McGee  \cite{MR1426482} extends  the observation to $L_{\infty\infty}$: 

\begin{quote}
Since the primitive connectives of $L_{\infty\infty}$  are all intuitively clearly logical connectives, and since, intuitively, anything definable from logical connectives is again a logical connective, this will show that every operation invariant under permutations is describable by a logical connective, so every operation invariant under permutations is a logical operation.

\end{quote} }

As was clear to Tarski, the criterion  {\em overgenerates}, in the sense that concepts one does not immediately assimilate to logic turn out to be included, such as the cardinality of the underlying domain, which is clearly isomorphism invariant. 
%As G. Sagi has argued, and as we will argue at length in this paper, first order logic emerges  again as having a high degree of logicality, via its strong L\"owenheim-Skolem theorem. This is because the theorem actually encodes the indifference to cardinality of first order logic, in the special sense that models of cardinality greater than $\aleph_0$ do not affect its validities.  
A second  problem  has to do with  domain relativity: Tarski's criterion identifies the logical operations on a fixed domain,  but that a domain consists of this or that type of object should not have anything to do with logicality---shouldn't a logical concept  be domain independent? G. Sher \cite{MR1203778}  repaired this problem by extending Tarski's criterion to cover  notions invariant across  isomorphic structures,  and accordingly the criterion is now known as the Tarski-Sher invariance criterion. 

%Most of these arguments  are characterised by the attempt to pin down the core notion (of logicality), so  ignoring Tarski\index{Tarski, Alfred}'s caveat. 

The Tarski-Sher criterion has generated a substantial literature, both for and against. 
 S. Feferman's \cite{MR2666567}, based on his earlier \cite{MR1811202},  proposes a notion of {\em homomorphism} invariance. His critique of  the Tarski-Sher thesis is the following:

\begin{quote}
I critiqued the Tarski-Sher thesis in \cite{MR1811202} on three grounds, the first of which is that it assimilates logic to mathematics, the second that the notions involved are not set-theoretically robust, i.e. not absolute, and the third that no natural explanation is given by the thesis of what constitutes the same logical operation over arbitrary basic domains.\footnote{\cite{MR2666567}, p. 1.}
\end{quote}

Feferman's suggestion that Tarski assimilates logic to mathematics,  in particular to set theory, sets aside Tarski's lifelong program to carry out  exactly what Feferman criticizes  Tarski of here, namely expressing metamathematical concepts in \mal and set-theoretical terms.\footnote{The first author has  argued this point in \cite{Kennedy2021-KENGTA-2}.}  Feferman's second objection has to do with overgeneration, in particular he resists the idea that a canonically \mal but non-absolute notion such as cardinality  is rendered  logical under the criterion.\footnote{The fact that cardinality is construed as logical under his criterion seemed to pose no problem for Tarski:

\begin{quote}
That a class consists of three elements, or four elements \ldots that it is finite, or infinite---these are logical notions, and are essentially the only logical notions on this level.  \cite{MR868748}, p. 151.
\end{quote}}

One way of thinking about  topic neutrality  is in terms of {\em absoluteness}, by which we mean general independence from the background set theory.  As is well-known, a set may have one cardinality in one model of set theory but another cardinality in another model.\footnote{Cardinality also emerges as an  artifact in the context of ramsification and Newman's objection to epistemic structural realism. See Ainsworth \cite{MR2485506}.} 
This shows that cardinality is not an absolute notion. On the other hand, whether a given set is empty or not, whether it is included in another given set, and whether it is the cartesian product of two given sets are all independent of the (transitive) model in which such questions are evaluated. Thus those notions are considered absolute. A possible reason to consider absoluteness a necessary quality of logical notions is that whatever depends on set theory could be considered a mathematical rather than logical notion. 

Quine has extensively argued against taking set-theoretic notions as logical,\footnote{See for example  \cite{MR844769}.} and in his seminal paper  \cite{MR868748} Tarski raises the question whether {\em mathematical} notions are logical. According to Tarski, if we construe mathematical notions in the framework of higher order logic, they are logical, while if we use the set-theoretical framework, based on a single binary relation $x\in y$, mathematical notions are not logical. Tarski concludes that the question of logicality of mathematical notions is unresolved and depends on how we construe mathematics. In this paper we take the set-theoretical approach to mathematics. However, we operate with model classes, namely  classes of structures of the same similarity type, but possibly of different cardinality, that are closed under isomorphism,\footnote{If we work in set theory, classes are objects of the form $\{a:\phi(a)\}$. Ordinary set theory does not have objects of this kind, so classes are simply identified with their defining formulas, in this case $\phi(x)$. The defining formula is allowed to have set parameters. Thus when we refer to the class of all ordinals, the class of all structures, etc, we mean to refer to the defining formula, which defines when a set is an ordinal, the formula which defines when a set is  structure, etc. If we worked in class theory, such as the Mostowski-Kelley-Morse class theory, we would not need to assume that classes are definable.}  which are part of the higher order logic framework. For example, according to Tarski the notion of a binary relation $R$ being a well-order is non-logical, as it is defined by reference to the $\in$-relation of set theory,  but the notion of a structure $(M,<)$ being a well-order is logical, as it is second order definable.  This example shows how delicate the question of logicality of mathematical notions is.
 %A logical notion, it is thought, should not be sensitive to the background set theory, in particular a logical notion should not be entangled with cardinality, which is famously non-absolute%
%JV
%, because of forcing

%Accordingly 
Feferman proposes restricting the Tarski-Sher invariance criterion to operations that are absolute with respect to set theories making no assumptions about the size of the given domain. Feferman cites \cite{MR1330511},\footnote{Feferman also cites an unpublished result of K. Manders.} in which it is shown that operations on relational structures that are definable in an absolute way relative to KPU-Inf, i.e. Kripke-Platek set theory with urelements and
without the Axiom of Infinity, are exactly those expressible in  ordinary first-order predicate calculus with equality---another endorsement of the idea that first order logic  captures the notion of ``logical concept."

Feferman's solution solves the overgeneration problem, while giving a nice characterisation of first order logic along the way. But this comes at the expense of a strong absoluteness assumption, and for  restricted classes of structures. 

Setting overgeneration aside for the moment, and considering Feferman's third objection, namely the problem of domain relativity, it would seem that following Feferman's line we should consider model classes in general.
How then to formulate the concept of logicality for a model class? Originally the question of logicality arises in connection with operations such as connectives, quantifiers and their generalisations. As V. McGee explains in  \cite{MR1426482}, the question can be restated as a question about model classes,  as the question whether the property of a given model being in a given model class is a logical property of the model.\footnote{This property is essentially the same as deeming the associated generalised quantifier,  together with everything first definable from it, as logical. See below section \ref{mcg}.} In other words, any relevant operation $O$ can be (re)formulated as a property $P$ of models (see Section~\ref{ovsmc}) and then the logicality of $O$ translates into the question whether the property $P$ is a logical property of models. Typical examples of logical properties of a unary model $(M,A)$, where $A\subseteq M$, would be the non-emptyness of $A$, and the property of $A$ of being equal to $M$. On the side of operators these example would correspond to the existential and the universal quantifers.
%In fact 
%it is rather generally recognised that first order definable properties, such as 
%$$\exists x(R(x,x)\wedge\neg P(x))$$
%$$\forall x\exists y R(x,y)$$ 
%should be considered logical. REF Bonnay and Feferman. SAY MORE HERE??? IS THIS IN THE RIGHT PLACE??

Let's back up and ask when would we call a property of an {\em individual} structure $\mm$ logical. With Sher, certainly we should consider the elements of $\mm$ irrelevant as long as the arrangement of the elements remains the same, as was mentioned. It would then be natural to call a property of $\mm$ logical only if the property is closed under permutations of the elements of $\mm$ in the sense that if $f$ is a permutation of $M$ then the image of $\mm$ as a structure under $f$ has the property.  But then it is a short step to require  closure under isomorphisms.

%In this paper we shall explore this question. We will consider a

A theorem due to McGee  \cite{MR1426482}  characterises logicality for an arbitrary model class, provided the cardinality of the models in the class is fixed in advance. It is a weakness of this characterisation that it depends on the cardinality of the models in the class. One might then ask, whether there is a  sentence in a logic $L^*$, perhaps other than $L_{\infty\infty}$, which characterises the logical property, not just in one cardinality, but in many cardinalities. This leads  to an analysis of {\em spectra}, as we call them,\footnote{See also \cite{MR3794871}.} and L\"owenheim-Skolem style properties of logics as expressed in their L\"owenheim and Hanf numbers, which we now define:

\begin{definition}\label{lnum}
Suppose $C$ and $D$ are classes of cardinal numbers.
 A logic $L^*$ satisfies the \emph{L\"owenheim-Skolem Property} $\LS(C,D)$ if every sentence in $L^*$
 % with vocabulary of \marginpar{What is $\kappa$?}size $\le\kappa$ 
 which has a model of some cardinality in $C$ has a model of some cardinality in $D$. 
 The smallest  $\kappa$ such that $L^*$ satisfies $\LS([1,\infty),[1,\kappa])$ is called the L\"owenheim number of $L^*$ and denoted $\ell(L^*)$. 
The smallest  $\kappa$ such that $L^*$ satisfies $\LS([\kappa,\infty),[\lambda,\infty))$ for all $\lambda$ is called the Hanf number of $L^*$ and denoted $h(L^*)$.
\end{definition}

As G. Sagi notes in her \cite{MR3794871}, the {\em L\"owenheim number} of a logic encodes the degree to which the logic is indifferent to cardinality.
%\footnote{For the definition of L\"owenheim number see section \ref{spe}. This observation is an  important point of departure for  this paper.}
%In detail, if a logic has a  strong L\"owenheim-Skolem style property, this blunts the effect of overgeneration (for that  logic).
This is due to the fact that if a sentence satisfied by a model is also satisfied by a model of size  less than or equal to a given cardinality $\kappa$, then all the validities\footnote{Or for Sagi, the meaning, see \cite{MR3794871}.} which are captured in the full range of models are captured already on an initial segment (of $V$). In the case of first order logic the  L\"owenheim-Skolem Theorem tells us that $\kappa$ can be taken to be $\aleph_0$, so in this sense first order logic is {\em indifferent} to cardinalities above $\aleph_0$. First order logic is already classified  as logical because its logical constants are permutation invariant; this is now witnessed by the L\"owenheim-Skolem Theorem.

The strength of the L\"owenheim-Skolem property of a given logic---or if you like, the degree of its indifference to cardinality---is measured, then, by its L\"owenheim number. As an example, the quantifier $Q_{\alpha}$ associated with the logic $\mathcal{L}(Q_{\alpha})$, having L\"owenheim number $\aleph_{\alpha}$, is in a sense indifferent to  cardinalities greater than $\aleph_{\alpha}$.\footnote{This is a central example of \cite{MR3794871}.}  First order logic and $\mathcal{L}(Q_0)$ are maximally logical under the criterion assigning logicality according to L\"owenheim number, and the degree of logicality decreases as $\alpha$ increases. Thus if $\alpha \leq \beta$, then $Q_{\alpha}$ is more logical than $Q_{\beta}$.\footnote{$\mathcal{L}(Q)$ denotes first order logic with the generalised quantifier $Q$ appended. The expression ``$Q_\alpha x\phi(x)$" means that ``there are at least $\aleph_\alpha$ many $x$ such that $\phi(x)$." By the L\"owenheim-Skolem Theorem, the  L\"owenheim number of first order logic is $\aleph_0$, as was noted; for each $\alpha$, the  L\"owenheim number of the logic $\mathcal{L}(Q_{\alpha})$ is $\aleph_{\alpha}$.}   

%Thus, Sagi suggests,  the degree of logicality is the stronger the  stronger form of a L\"owenheim-Skolem Theorem holds for the logic.

%Together with McGee's Theorem, Sagi's suggestion to tie logicality to L\"owenheim numbers is our point of departure in this paper. 
The criterion for logicality presented in Sagi's paper is based on a particular philosophical position, to wit: a metaphysical  view of the cumulative hierarchy of sets, as well as a view about how the meaning of the terms of a logic are fixed.  
  \begin{quote}
\ldots 
%for a logic $L$, if $\ell(L)= \kappa$, then dismissing all models of cardinality greater than $\kappa$ will not make a difference to the validities of the logic. the 
the higher set-theoretic infinite is metaphysically loaded\ldots The lower the cardinalities to which the {\em meaning}\footnote{emphasis ours} of a term may be sensitive, the more logical it is.  \ldots [We] view L\"owenheim numbers as telling us how much structure a term [in this case the quantifiers $Q_\alpha$ JK/JV] requires in order to be fixed in the context of a logic."\footnote{\cite{MR3794871}} 
%\ldots We may venture to say that the terms in the logic $L$ in such a case, if fixed faithfully, require no more structure than that which is given by the set of cardinalities less or equal to  $\kappa$.   \footnote{REF} 
\end{quote}

In short, the smaller the  L\"owenheim number of the logic,  the less metaphysically involved e.g. the quantifier  $Q_{\alpha}$ is.

Sagi's criterion partitions  logical space in  ways that differ from ours. In this paper we will not take account of the possible metaphysical commitments of set theory;  and while her analysis of  meaning is one we potentially endorse, we will not take a stand on the meanings of the terms of a logic in this paper. Our basic thesis is simply this: {\em if first order logic is taken to be the fundamental exemplar  of logicality, then logics that resemble, to a degree, first order logic in their model theoretic properties should be graded as logical to that degree}. 

Why do we take first order logic to be the fundamental exemplar  of logicality? Considerations of space do not allow us to argue for this contested point here.\footnote{ Those who have contested  the view that logic is first order logic include Sher \cite{MR1203778} and Barwise \cite{MR819532}.  Many others have argued for the value of nonclassical and higher order logics on grounds other than logicality such as Boolos \cite{boolos}. See also Sher's 2016 \cite{sherfriction}. } Indeed those opposed to the thesis that first order logic is maximally logical can perhaps find this paper useful as a test of that very thesis, given the complexity of the landscape of higher order logics laid out here, with respect to their (graded) degrees of  logicality. For example, we suggest in this paper that logics such as $\mathcal{L}(Q_1)$ have, arguably, a higher degree of logicality than $\mathcal{L}(Q_0)$, on the basis of the completeness of the former with respect to the Keisler axioms (see (\ref{ka}) in Section~\ref{compl}).

As is well known, Quine was a forceful advocate for the thesis that logic is first order logic on various grounds, including: First, the fact that first order logic has a complete proof system. Second, first order logic is the basis of  the distinction between logic and set theory, as we mentioned above. Finally, there is Quine's view of semantics, in particular his view that first order quantification is the optimal instrument for reading off the ontological commitments of a given theory, being itself ontologically minimal.\footnote{See e.g.  \cite{MR844769}.} 

Apart from completeness, we do not take a stand on these stalwarts of Quinean logical theory here,\footnote{More recent advocates for the thesis that first order logic is logic include Bonnay \cite{MR2395046} and Feferman \cite{MR2666567}.} rather referring to what seems to be a consensus among writers on the topic  that the connectives ($\wedge,\vee,\neg,\to$
, etc) as well as the quantifiers ($\exists,\forall$) are logical operations, being closed under bijections, and also on the basis of what seems to be a strong intuition that they have the same meaning across domains of even different sizes. In fact, in many cases  these operations are where the discussion on logicality starts. First order logic consists of iterations of these operations. Why would logicality disappear in this iteration? When we go beyond first order logic, e.g. to infnitary connectives and generalized quantifiers, we recognise the emerging problem whether we remain in the realm of logicality.

Thus while  taking up Sagi's suggestion that logicality may be calibrated by  L\"owenheim numbers, we also give weight to the other model theoretic properties of a logic, such as whether they have a Completeness Theorem or not.\footnote{Quine  emphasises the Completeness Theorem as primary evidence of the logicality of first order logic, asserting it as proof of the existence of ``an integrated domain [i.e. first order logic JK/JV] of logical theory with bold and significant boundaries."  See p. 90-91 of \cite{MR844769}.} As we will point out below,  logics that have, for example, a Completeness 
Theorem do not seem to be tied to the $\aleph$-hierarchy in any obvious way.
%---some of these logics are axiomatisable and some are not. 
%ßFor us, then,  the $\aleph$-hierarchy  doesn't serve as straightforwardly  calibrating logicality.

Another drawback of McGee's theorem which we address in this paper has to do with the model-theoretic properties of the logic $L_{\infty\infty}$, namely that it is highly non-absolute (see Definition~\ref{abso}), that it fails to have a strong L\"owenheim-Skolem theorem, and also that it is \emph{unbounded\footnote{as it is defined in \cite{zbMATH03506668}}} in the sense that it can define the concept of well-ordering leading to very large Hanf numbers.\footnote{See discussion following (\ref{wf}).} This leads us to ask whether $L_{\infty\infty}$ can be replaced by a ``tamer" logic, one closer to being first order in its model-theoretic properties. If such were to exist, then even with the problem of dependence on the cardinality of the models in the class unsolved, the relevant logicality claim  would be strengthened by virtue of its  proximity to first order logic.

In sum: in this paper we develop further this aspect of logicality identified by Sagi, namely that it is graded, albeit with a different philosophical agenda
%\todo{Maybe add here ref to Sher's book  \cite{MR1203778}.}
than is laid out in  \cite{MR3794871}. 
We especially point out problems in the L\"owenheim-Skolem properties of $L_{\infty\infty}$, along with other ways in which  it diverges from first order logic. As was mentioned above, we take first order definability as a particularly strong form of logicality. Accordingly, our suggestion here is that definability in a logic which resembles with its model theoretic properties first order logic should represent an intermediate form of logicality. We calibrate these intermediate forms according to the model theoretic properties of these logics, considering mainly  absoluteness, having a L\"owenheim-Skolem theorem, and having a completeness theorem.

While $L_{\infty\infty}$ is sufficient for McGee's theorem, there is a weaker logic, namely the so-called $\Delta$-extension of $L_{\infty\omega}$ which does the job as well. We shall explain in which way this weaker logic is better than  $L_{\infty\infty}$ as a test of logicality, and also point out in which respects it  may be lacking. Looking ahead, our refinement of McGee's theorem  replaces  $L_{\infty\infty}$ by a logic which is absolute and which has a  good L\"owenheim-Skolem theorem, along with a number of other desirable properties. The improvement here is partial in the sense that it is still the case that the definition is given relative to the size of the models in the class. In section \ref{sort} below we give a theorem which does not rely on this assumption, namely  we present a logic in which any model class is definable irrespective of the cardinality of the models in the class.

In spite of the seemingly clear intuition behind the syllogistic example with which we opened this paper, logicality is an elusive concept. It is by no means obvious what it should mean, even informally. But then, without such a clear informal intuition of its meaning it is difficult to judge whether this or that improvement of Tarski's original definition hits the mark or not, or even whether it is  a step in the right direction. The only thing there seems to be a consensus on in the wake of Tarski's suggestion to identify logicality with isomorphism invariance is that Tarski's criterion, extended by Sher, is a necessary but not sufficient one.\footnote{As Bonnay and Westerst\aa hl put it in their \cite{MR3531783}: ``And topic-neutrality, in the precise form of invariance under permutations of the universe, is almost universally agreed to be a necessary condition for logicality. It guarantees that the logical core of a language is general enough to carve out content
in any conceivable situation of language use, irrespective of what objects are being talked about."} 
 
 In this paper we create a base map of the landscape relevant for logicality. Following Tarski's, ours is a semantic approach. Beginning with McGee's Theorem \cite{MR1426482} to the effect that logical operations in the sense of Tarski-Sher can be described in each cardinality separately in $L_{\infty\infty}$, the landscape that opens in front of us is the world of different logics and the question, do the model-theoretic properties of these logics shed any light on the logicality problem?

\section{McGee's Theorem} \label{mcg}

To describe McGee's result and to put it into a more general framework we define what we mean by a ``logic": A \emph{logic} (a.k.a. abstract logic) in the sense of \cite{MR0244013} is a pair $L^*=(\Sigma,T)$, where $\Sigma$ is an arbitrary set (sometimes also a class) and $T$ is a binary relation between members of $\Sigma$ on the one hand and structures on the other. Members of $\Sigma$ are called $L^*$-sentences. Classes of the form $$\Mod(\phi)=\{\mm:T(\phi,\mm)\},$$ where $\phi$ is an $L^*$-sentence, are called $L^*$-characterizable, or $L^*$-definable, classes. Abstract logics are assumed to satisfy five axioms expressed in terms of $L^*$-characterizable classes, corresponding to being closed under isomorphism, conjunction, negation, permutation of symbols, and ``free" expansions.\footnote{The free expansion to vocabulary $L$ of a model class $K$ of a smaller vocabulary is  the class of all expansions of elements of $K$ to the vocabulary $L$.}  A class $\K$ of models is said to be \emph{definable} in a logic $L^*$ if there is a sentence $\phi$ in $L^*$ such that
$$\K=\Mod(\phi).$$

Every model class is definable in \emph{some} logic because we can take the model class as a generalized quantifier in the sense of \cite{MR0244012}: Suppose $\K$ is a model class with vocabulary $L$. For simplicity we assume $L=\{R\}$ where $R$ is a binary predicate symbol. We can associate with $\K$ the generalized quantifier $Q_\K$ with the semantics
$$\mm\models Q_\K xy\phi(x,y,\vec{a})\iff (M,\{(b,c)\in M^2 : \mm\models\phi(b,c,\vec{a})\})\in\K.$$ Now $\K$ is trivially definable in the extension $\Lo_{\omega\omega}(Q_\K)$ of first order logic by the quantifier $Q_\K$ by the sentence
$$Q_\K xyR(x,y).$$ Conversely, every class of models definable in $\Lo_{\omega\omega}(Q_\K)$, or indeed in any abstract logic, is a model class i.e. is closed under isomorphisms. We obtain the following simple and at the same time basic characterization:

\begin{theorem}[\cite{MR0244012}]\label{lind}
If $\K$ is a class of models of the same  vocabulary, then the following conditions are equivalent:
\begin{enumerate}
\item $\K$ is closed under isomorphisms i.e. $\K$ is a model class.
\item $\K$ is definable in some extension of first order logic by a generalized quantifier.
\item $\K$ is definable in some logic.
\end{enumerate}\end{theorem}

%\section{Potential isomorphism as an alternative to isomorphism}

\subsection{Operations vs. model classes}\label{ovsmc}

In the literature, logical operations and model classes are considered to be carriers of logicality. Our paper focuses on model classes and on the extent to which they can be called logical, but here we take a moment to establish the connection between model classes and logical operations. McGee \cite{MR1426482} defines an abstract concept of what he calls a \emph{logical operation}. He goes on to define what it means for a logical operation to be \emph{described} by a formula of a logic. We review these definitions and establish a close connection between model classes and connectives as well as between definability of a model class and describability of a logical operation.

Suppose $M$ is a non-empty set. By the  semantic value (on $M$) of a formula (of any logic) we mean the set of assignments into $M$ that satisfy the formula. Abstractly, these are just  subsets $A$ of $M^{n}$ for some $n$. A \emph{(local) operation} $f$ \emph{on} $M$  maps  sequences $\langle A_\alpha : \alpha<\beta\rangle$ of  sets $A_\alpha\subseteq M^{n_\alpha}$ to  sets $f(\langle A_\alpha : \alpha<\beta\rangle)\subseteq M^n$. Such a local operation is a \emph{logical operation}, if it is closed under permutations of $M$, i.e. if  for all permutations $\pi$ of $M$
$$f(\langle \pi''A_\alpha : \alpha<\beta\rangle)=\pi''f(\langle A_\alpha : \alpha<\beta\rangle),$$
where for any $s=(a_1,\ldots,a_n)\in A\subseteq M^n$ we define $\pi(s)=(\pi(a_1),\ldots,\pi(a_n))$. 

\begin{example}Suppose $A_0,A_1 \subseteq M^n$. \emph{Conjunction} is the logical operation
$f^n_\wedge(\langle A_0,A_1\rangle)=A_0\cap A_1$. \emph{Disjunction} is the logical operation
$f^n_\vee(\langle A_0,A_1\rangle)=A_0\cup A_1$. Suppose $A\subseteq M^{n+1}$. The
\emph{existential quantifier} with respect to $n$ is the logical operation
$f^{n}_{\exists}(\langle A\rangle)=\{s\restriction\{0,\ldots,n-1\}:s\in A\}$.
\end{example}

For the existential second (and higher) order quantifiers to be logical operations would require an extension of the approach to include more general assignments, which we however disregard in this paper. Still, every formula of second (or higher) order logic gives rise (separately) to a logical operation in the above sense. The situation is the same with the so-called team semantics \cite{MR2351449} where semantic values are sets of sets of assignments, rather than sets of assignments as above.

More generally we  have an operation $f_M$ for every non-empty set $M$. Here we assume that $f_M$ is always an operation on $M$. We call the (class) mapping $M\mapsto f_M$ a \emph{global operation} and denote it $\bar{f}$. We call the global operation $\bar{f}$ a \emph{logical operation} if it is  closed under bijections, i.e.  
if  for all bijections $\pi:M\to N$
$$f_N(\langle \pi''A_\alpha : \alpha<\beta\rangle)=\pi''f_M(\langle A_\alpha : \alpha<\beta\rangle),$$

A formula $\phi(\vec{x})$ with an $n_\alpha$-ary predicate symbol $P_\alpha$ for each $\alpha<\beta$ is said to \emph{describe}  a local operation $f$ on $M$ if  in any model  $\mm$ with domain $M$ and $P_\alpha^\mm=A_\alpha\subseteq M^{n_\alpha}$  the semantic value of $\phi(\vec{x})$ is $f(\langle A_\alpha : \alpha<\beta\rangle)$. The concept of a formula describing a global operation is defined similarily.

\begin{example}
The formula $P_0(\vec{x})\wedge P_1(\vec{x})$ describes the operation $f^n_\wedge$.
The formula $P_0(\vec{x})\vee P_1(\vec{x})$ describes the operation $f^n_\vee$.
The formula $\exists x_n P_0(\vec{x},x_n)$ describes the operation $f^n_\exists$.\end{example}
If $\phi(\vec{x})$ is a formula of a logic $L^*$ in the above sense, then the operation described by $\phi(\vec{x})$ is a logical operation, since we assume that truth in every logic $L^*$ that we consider is closed under isomorphisms (cf. \cite{MR0244013}).

 An operation $\bar{f}$ can be represented alternatively as the model class in the vocabulary $L$ which has in addition to the predicate symbols $P_\alpha$, $\alpha<\beta$, a new predicate symbol $P$: $$\K_{\bar{f}}=
 \{\mm: \mbox{$\mm$ is an $L$-structure and }P^\mm=f_M(\langle P^\mm_\alpha:\alpha<\beta\rangle))\}.$$
%
%
%\begin{lemma}
It is easy to see that the class $\K_{\bar{f}}$ is closed under isomorphisms if and only if  the operation $\bar{f}$ is preserved by bijections.

On the other hand, any model class $\K$ can be turned into a logical operation $\bar{f^\K}$ as follows. Let for any non-empty set $M$ and relations $\vec{R}$ on $M$ of the right arity:
$$f^{\K}_M(\vec{R})=\left\{\begin{array}{ll}
M&\mbox{if $(M,\vec{R})\in\K$}\\
\emptyset&\mbox{otherwise.}
\end{array}\right.$$
%
%%
%\begin{lemma}
It is easy to see that the class $\K$ is closed under isomorphisms if and only if the operation $\bar{f^\K}$ is preserved by bijections.

\subsection{Cardinal dependent definability}

The theorem  of McGee improves ``some logic" of Theorem~\ref{lind} to a very specific logic, namely $L_{\infty\infty}$, but at a price: {\em we obtain a different definition  for each cardinality separately}.\footnote{We point out other weaknesses in the  below section~\ref{eww} involving the model-theoretic properties of $L_{\infty\infty}$, e.g. its non-absoluteness.} 
% In other words, if a property of $\mm$ is closed under isomorphisms, then the property can be expressed in $\Lo_{\infty,\infty}$ as a property of models of the same size as $M$ and vice versa. 

Note that Theorem~\ref{lind} has some  uniformity in the sense that the defining sentence does not depend on the cardinality of the model in question, but it is also lacking in uniformity in the sense that we do not have one single logic but possibly a different logic for each model class.

We now isolate this property of cardinality dependence.  For any cardinal $\lambda$ let $\K_\lambda$ be the class of elements of $\K$ with a domain of size $\lambda$. 
%Let $\Str$ denote the class of all structures (in a given vocabulary).

\begin{definition}
A model class $\K$ is \emph{cardinal dependently definable}, or \emph{CD-definable}, in a logic $L^*$, or \emph{cardinal dependently $L^*$-definable}, if $\K_\lambda$ is $L^*$-definable for every $\lambda$.
\end{definition}

Note that a model class can be   CD-definable even in first order logic without being definable even in $L_{\infty\infty}$:  Consider the class $\K_{\mbox{\tiny lim}}$ of models $(M,P)$, where either $|M|$ is a limit cardinal and $P=\emptyset$ or else $|M|$ is a successor cardinal and $P\ne\emptyset$. It is a consequence of the L\"owenheim-Skolem Theorem of $L_{\infty\infty}$ (see Proposition~\ref{lim} and the proof of Theorem~\ref{har} below) that this model class cannot be definable in it.  The following theorem relates isomorphism closure to CD-definability in the case of classes of models:

\begin{theorem}[\cite{MR1426482}]
\label{mcgee}
If $\K$ is a class of models of the same  vocabulary, then the following conditions are equivalent:
\begin{enumerate}
\item $\K$ is closed under isomorphisms.
\item $\K$ is CD-definable in $L_{\infty\infty}$.
\end{enumerate}
\end{theorem}

\begin{proof}
Let us fix an infinite cardinal $\lambda$. Let $\mm_\alpha$, $\alpha<2^\lambda$, list all elements of $\K$ with domain $\lambda$. Let $\theta_\alpha\in L_{\lambda^+\lambda^+}$ characterize up to isomorphism the model $\mm_\alpha$. The sentence $\bigvee_{\alpha<2^\lambda}\theta_\alpha$ defines the model class $\K_\lambda$.
\end{proof}

The fact that McGee's  theorem depends on the cardinality of the models in the class means that the sentence of $\phi_\lambda$ of $L_{\infty\infty}$ defining the model class $\K_\lambda$ may very well depend on $\lambda$, as in the above example $\K_{\mbox{\tiny lim}}$. That is, the class size mapping $\lambda\mapsto\phi_\lambda$ may encode information that has to be analysed as to its logicality.  In the above example, when we ask whether the property of a model of belonging to the model class $\K_{\mbox{\tiny lim}}$ is logical or not, we essentially ask whether the property of a model of having a limit cardinality is logical or not. According to the Tarski-Sher criterion it undoubtedly is logical. On the other hand, it is fair to say that it is a mathematical property rather than a logical one. As we will see below, the logicality of membership in $\K_{\mbox{\tiny lim}}$ manifests in our below analysis a rather low degree of logicality.

As we pointed out earlier, McGee \cite{MR1426482} considers describability of logical operations rather than definability of model classes. Let us now draw a further connection,  using the $\Delta$-operation on logics, between describability of logical operations and definability of model classes. 
%We refer to Section~\ref{abs} for the definition of $\Delta(L^*)$. 

%We will now define the $\Delta$-extension of a logic. 
Suppose $\mathcal{K}$ is a class of models of vocabulary $L$ and $L'\subseteq L$. We use $\K\restriction L'$ to denote the class of reducts $\mm\restriction L'$ of models $\mm\in \K$. We call $\K\restriction L'$  a {\em projection} of $\K$. 

\begin{definition}\label{sigma11}
A model class is $\Sigma(L^*)$-definable, or in $\Sigma(L^*)$, if it is a projection of an $L^*$-definable model class. A model class is $\Delta$-definable in $L^*$, or in $\Delta(L^*)$, if both $\K$ and the complement of $\K$ are $\Sigma(L^*)$. 
\end{definition}

It should be noted that $\Delta(L_{\omega\omega})=L_{\omega\omega}$ \cite{MR104564}, $\Delta(L_{\omega_1\omega})=L_{\omega_1\omega}$ \cite{MR188059} and $\Delta(L_{\infty\omega})\subseteq L_{\infty\infty}$ \cite{MR290943}.

\begin{lemma}Suppose $f$ is a logical operation on $M$ and $L^*$ is a logic. Then (1) $\to$ (2) $\to$ (3), where:
\begin{enumerate}
\item[(1)] $f$ is describable in  $L^*$. 
\item[(2)] The model class $\K_f$ is definable in $L^*$. 
\item[(3)] $f$ is describable in $\Delta(L^*)$.

\end{enumerate}
\end{lemma}

\begin{proof}Suppose $f$ is described on $M$ by the formula $\psi(\vec{x})\in L^*$.
$\K_f$ is defined by the sentence $\forall\vec{x}(P(\vec{x})\leftrightarrow\psi(\vec{x},\vec{R})).$ This proves the first implication. Suppose then $\K_f$ is defined by the sentence $\phi\in L^*$ with vocabulary $\{P_\alpha:\alpha<\beta\}\cup\{P\}$. Consider the formula $\psi(\vec{x})\equiv\exists P(\phi\wedge P(\vec{x}))$ with the second order quantifier $``\exists P"$. This clearly describes $f$ but is not, a priori, in $L^*$. However $\psi(\vec{x})\equiv\forall P(\phi\to P(\vec{x}))$ and therefore $\psi(\vec{x})$ is in $\Delta(L^*)$.
\end{proof}

\begin{lemma}Suppose $\K$ is a model class. Then  $\K_\lambda$ is definable in  $L^*$ if and only if 
the logical operation $f^{\K}_\lambda$ is describable in $L^*$. 
\end{lemma}

\begin{proof}Suppose $\K_\lambda$ is defined in cardinality $\lambda$ by the sentence $\phi\in L^*$.
$f^{\K}_\lambda$ is defined by the formula $\phi$ (without free variables).  Suppose then $f^{\K}_\lambda$ is defined by the formula $\psi({x})\in L^*$. Then $\K_\lambda$ is defined by the sentence
$\exists {x}\psi({x})$.
\end{proof}

The above two lemmas demonstrate the close relationship between logicality of operations on semantic values and logicality of properties of models. The next lemma, based on the method of \cite[Theorem 5]{MR0244013}, shows that the relationship is not perfect.

\begin{lemma}
There is a logical operation $\bar{f}$ such that $\K^{\bar{f}}$ is definable in $L(Q_0)$ but $\bar{f}$ is not describable in
$L(Q_0)$. 
\end{lemma}

Bonnay \cite{MR2395046} suggests that the concept of potential isomorphism is a better criterion for logicality than isomorphism itself. 
For potential isomorphism closure there is  the following version of McGee's Theorem:

\begin{theorem}[\cite{MR0342370}]\label{baarwise}
If $\K$ is a class of models of the same  vocabulary, then the following conditions are equivalent:
\begin{enumerate}
\item $\K$ is closed under potential isomorphisms.
\item $\K$ is CD-definable in $L_{\infty\omega}$.
\end{enumerate}
\end{theorem}

%Of course, logicality across domains remains a problem even if we get rid of cardinalities; that is, potential isomorphism does not settle Feferman's third objection to the Tarski-Sher criterion, that it leaves unexplained how logicality is to be understood across domains of different sizes.

%Another 
A problem with closure under potential isomorphism is that well-ordered structures
$(\alpha,\in,P_1,\ldots,P_n)$ and $(\beta,\in,P'_1,\ldots,P'_n)$ 
are never potentially isomorphic if $\alpha\ne\beta$. Just as an isomorphism closed model class can have models with domains of different cardinalities and the isomorphism closure gives no information across the domains, a potential isomorphism closed model class can have models $(\alpha,\in,P_1,\ldots,P_n)$ with different $\alpha$ and the potential  isomorphism closure also gives no information across such domains.

\subsection{The model-theoretic properties of $L_{\infty\infty}$ }\label{eww}

We now note the different respects in which  $L_{\infty\infty}$ deviates from first order logic. First of all, $L_{\infty\infty}$ is badly nonabsolute.\footnote{For the technical definition of absoluteness see Section~\ref{abs}.} For example, the sentence
\begin{equation}\label{nonabs}
\begin{array}{l}
\forall x\forall y(\forall z(E(z,x)\leftrightarrow E(z,y))\to x=y)\wedge\\
\forall x_0\forall x_1\ldots\exists y\forall z(E(z,y)\leftrightarrow\bigvee_{n<\omega}z=x_n)
\end{array}\end{equation} of $L_{\infty\infty}$ which has models exactly in cardinalities $\mu$ such that $\mu^\omega=\mu$, is highly non-absolute.\footnote{By K\"onig's Theorem (\cite[Theorem 5.10]{MR1940513}) such a $\mu$ cannot have cofinality $\omega$ and cofinality is not absolute in set theory.} Other failures of absoluteness can be generated by  replacing $\omega$ here by any regular cardinal.

A second important model-theoretic property of first order logic is its downward  L\"owenheim-Skolem theorem. Consider  the above (\ref{nonabs}). Such a $\mu$ cannot have cofinality $\omega$ and thus $L_{\infty\infty}$ does not have a L\"owenheim-Skolem theorem in the same strong sense as first order logic, or even in the sense of, e.g., $L_{\infty\omega}$ (see Theorem~\ref{ls}).  That is, $L_{\infty\infty}$ can ``omit" cardinals of cofinality $\omega$ in this special sense given by the L\"owenheim-Skolem theorem.  
%
%In an obvious way, having a L\"owenheim-Skolem theorem forges a connection between different cardinalities of models in definable model classes, if it might not even be said to collapse them into a single cardin. 
In contrast,  a typical consequence of a L\"owenheim-Skolem theorem of $L_{\kappa^+\omega}$ is that if a definable model class has a model of cardinality $\lambda>\kappa$ then it has models of all cardinalities $\mu$ such that $\kappa\le\mu\le\lambda$ (Theorem~\ref{ls}), whatever their cofinality.  
%For $L_{\kappa^+\kappa^+}$ this holds for cardinals $\mu$ with $\mu^\kappa=\mu$ only, again leaving out all $\mu$ with cofinality $<\kappa$.

The logic behind Theorem~\ref{mcgee}, $L_{\kappa^+\kappa^+}$,  fails in a strong sense to have a Completeness Theorem. By a result of Dana Scott (published in \cite{MR0176910}) the set of valid sentences of $L_{\kappa^+\kappa^+}$ is not $L_{\kappa^+\kappa^+}$-definable over $H(\kappa^+)$, the set of sets of hereditary cardinality $\le\kappa$. In contrast, the set of valid sentences of $L_{\kappa^+\omega}$ is $\Sigma_1$-definable over $H(\kappa^+)$ \cite{MR0176910}, reminiscent of the Completeness Theorem of first order logic which implies that the set 
of (G\"odel numbers of)  valid sentences of first order logic is recursively enumerable i.e. $\Sigma_1$ over $HF$, the set of hereditarily finite sets. In fact, C. Karp gives a complete axiomatization for $L_{\kappa^+\omega}$ in \cite{MR0176910}.

%????In an obvious way, having a L\"owenheim-Skolem theorem forges a connection between different cardinalities of models in definable model classes. If we think of the spectrum $\sp(\K)$ of a model class $K$, namely the class of cardinalities of the models in the class $K$, then an auxiliary  concept emerges here: that of the {\em L\"owenheim spectrum} of the class, $\sp(\K)_L$, namely the class of the cardinalities to which the sentence is not indifferent as given by the L\"owenheim-Skolem theorem. In the case of first order logic, the spectrum of models is very regular, while in the above .  Conversely, regular patterns in cardinalities of models in a model class can be considered indications that the model class is actually definable in some logic with a reasonable  L\"owenheim-Skolem theorem. HOW DOES A REGULAR SPECTRUM GIVE THESE INDICATIONS???? 

In section~\ref{abs} we replace  $L_{\infty\infty}$ by a logic which is (almost) absolute and which has a strong L\"owenheim-Skolem theorem, together with other desirable properties, though it is still the case that the definition is given relative to the size of the models in the class. 
%\marginpar{repetition with end of Sect. \ref{eww}?} In section \ref{sort} below we give a theorem which does not rely on this assumption, namely a we present a logic in which any model class is definable irrespective of the cardinality of the models in the class.  
%\medskip

%PUT this somewhere else: \noindent{\bf Notation:} We use $\mm,\mn$ etc to denote structures and the corresponding latin letters $M,N$ etc to denote the domains of the corresponding structure. The domain is always assumed to be a set. The cardinality of a set $M$ is denoted $|M|$.

\section{L\"owenheim-Skolem theorems viewed through the spectra of model classes}\label{spec}

The fact that the class of all models of cardinality $\kappa$, for example,  satisfies the Tarski-Sher criterion, irrespective of what $\kappa$ is, raises the question, is cardinality a logical property? We saw that in the more subtle case of $\K_{\mbox{\tiny lim}}$ the property of a model of being of a limit cardinality is, according to this criterion, logical; by modifying $\K_{\mbox{\tiny lim}}$ one can generate many other cases. For example, if $A$ is a property of natural numbers, we can consider the class $\K$ of models of cardinality $\aleph_n$, such that $n$ has property $A$.  The property of a model of belonging to $\K$ is a logical property of the model by the Tarski-Sher criterion, even though to judge whether a particular model is in $\K$ one has to first determine whether the size of the universe is some $\aleph_n$, and after that one has to determine whether $n$ has property $A$. This stretches the concept of logicality and entangles it with the concept of what is mathematical. Earlier writers have observed this, so we are not saying essentially anything new here. 

Of Feferman's three objections to the Tarski-Sher criterion of logicality, the third is that it leaves unexplained how logicality is to be understood across domains of different sizes. The existential quantifier, for example, is, in a sense, the same on any domain of any cardinality, but this is not true of many properties which satisfy the Tarski-Sher criterion, for example $\K_{\mbox{\tiny lim}}$. We suggest that in order that a property of models manifests a greater degree of logicality than is provided by the Tarski-Sher criterion, it should have a definition by a sentence in $L_{\infty\infty}$ or in some other logic in such a way that {\em the same sentence defines the model class in as many cardinalities as possible}.\footnote{Note that $\K_{\mbox{\tiny lim}}$ has the same first order definition on the (closed unbounded) class of \emph{all} limit cardinals.} We explore this suggestion below via the concept of a spectrum. Our second suggestion is that a higher grade of logicality would be assigned to the degree to which the logic in question had ``first-order-like" model-theoretic properties: absoluteness, a L\"owenheim-Skolem theorem, and completeness. $L_{\infty\infty}$ would thus be given a low grade on the logicality scale, being non-absolute,  having a limited L\"owenheim-Skolem theorem and failing to have a Completeness Theorem. We explore this suggestion in section~\ref{compl}\footnote{For a different analysis of domain relativity, see \cite{Westerstahl2017-WESS-5}.}.

%We now introduce  the concept of a spectrum.

\subsection{Spectra}\label{spe}
In analysing the structure of a model class, in order to determine whether it is in some sense logical or not or to find out whether it is definable in some interesting logic, it is useful to investigate the cardinalities of models in the class. One might think that mere cardinalities are too rough a measure of any form of logicality but this is, in fact, surprisingly informative.
This leads naturally  to the concept  of a spectrum:\footnote{See also \cite{MR3794871}.} 

\begin{definition}
If $\K$ is a model class, the \emph{spectrum} of $\K$ is the class $\sp(\K)$ of cardinalities of models in $\K$ i.e.  $$\sp(\K)=\{|M| : \mm\in \K\}.$$
\end{definition}

Depending on $\K$, the spectrum can be a singleton, an interval of cardinals, an initial (or final) segment of the class of all cardinals, or something more complicated, such as the class of all limit cardinals, or all limit cardinals of cofinality $\omega$ (see Figure~\ref{pic2}). Even the patterns of finite numbers in spectra of first order sentences is highly interesting \cite{MR77468,MR485317}. However, we are here concerned with infinite cardinals in a spectrum. 

\begin{figure}
\begin{center}
\includegraphics[height=5cm]{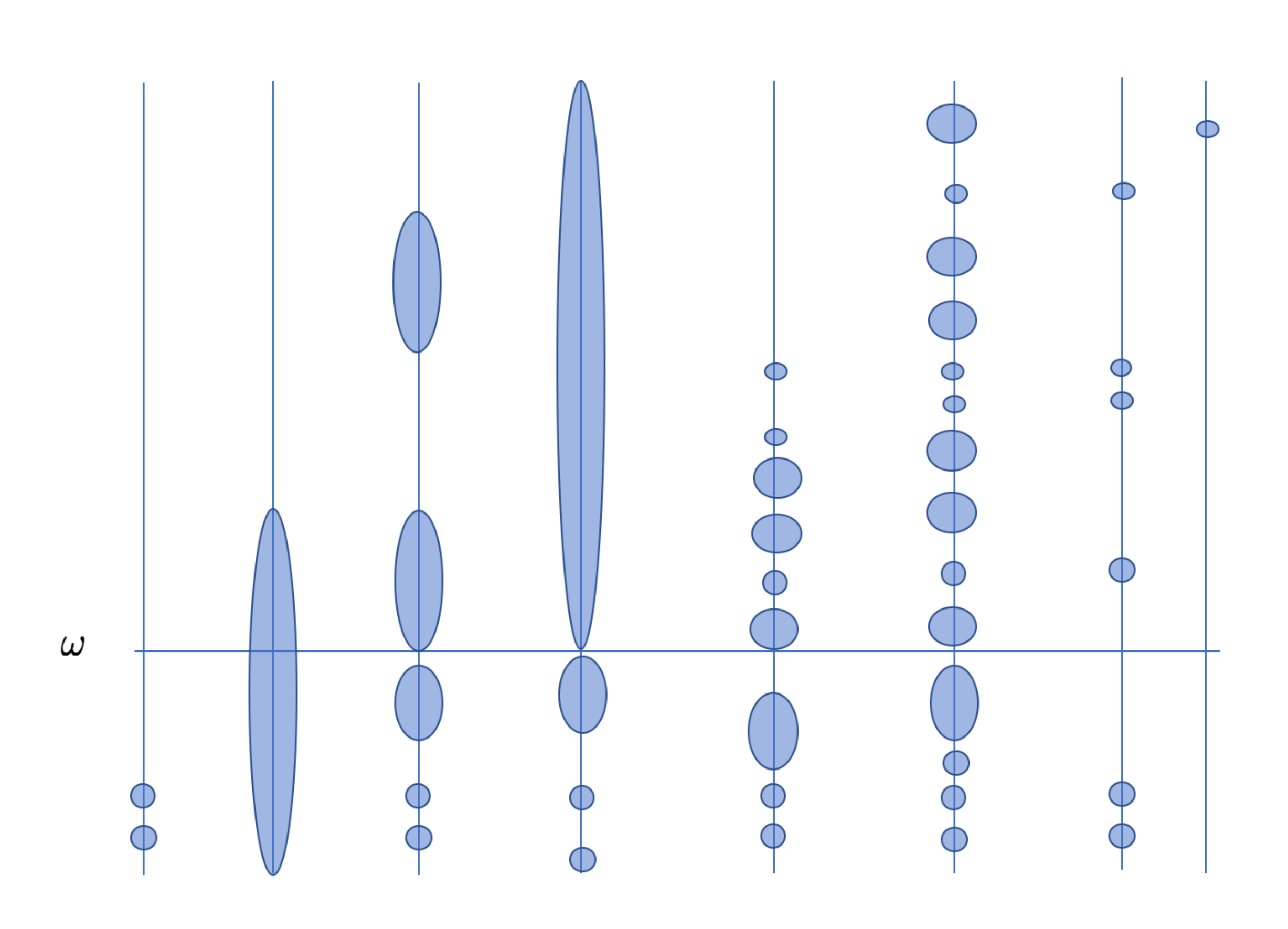}
\caption{\label{pic2}Spectra of some model classes.}
\end{center}
\end{figure}

The  property of a logic which reflects regularity patterns in its spectra is captured by the  L\"owenheim-Skolem Theorem.  The spectrum gives  indirect information of the possibility that the model class is definable in some logic. Roughly speaking, if the logic has a strong L\"owenheim-Skolem property,  then the spectra of definable model classes reflect this. If every sentence in the logic which has an infinite model has also a countably infinite model, the most famous case of a L\"owenheim-Skolem property, then every spectrum with an infinite cardinal in it has also $\aleph_0$ in it.

Skolem proved that first order logic satisfies $\LS([\aleph_0,\infty),\{\aleph_0\})$ (see Definition~\ref{lnum}) giving rise to the so-called Skolem Paradox: countable first order theories such as set theory have countable models if they have models at all. Now, one hundred years after Skolem's discovery, the indifference of first order logic to the infinite cardinality of its models is often considered a positive rather than negative aspect. Indifference to cardinality is of course closely related to set-theoretical absoluteness and thereby to a desirable quality of logicality. In consequence, the stronger form of $\LS(C,D)$ a logic satisfies the more appropriate the logic is for expressing logical properties.

If a logic satisfies $\LS(C,D)$, there are consequences for the spectra of definable model classes.  Suppose $\K$ is definable in a logic with $\LS(C,D)$. Then we can make the following conclusions: If there is $\mm\in\K$ with $|M|\in C$, then there is $\mn\in\K$ with $|N|\in D$. On the other hand, if $\K$ contains a model of cardinality $\kappa$ but no models of cardinality $\lambda$, then $\K$ cannot be definable in a logic with $\LS(C,D)$ such that $\kappa\in C$ and $\lambda\notin D$. The point is that by looking at the spectrum of $\K$ we can make inferences about its definability in different logics.
Thus, if we are given a model class $\K$ but no logic in which it would be definable 
(apart from the trivial $L(Q_\K)$), and we can discern regular patterns in $\sp(\K)$, we may take it as an indirect indication (albeit not a proof) that $\K$ is definable in a logic with $\LS(C,D)$ for some $C$ and $D$ explaining the found patterns. On the other hand, if $\sp(\K)$ does not seem to have regular patterns, we may take it as an indication that no such logic can be found. Thus $\sp(\K)$ gives implicit information about the possibility of finding a syntax for $\K$.

Let us look at the L\"owenheim-Skolem Properties of the logics $L_{\infty\omega}$ and $L_{\infty\infty}$.\footnote{Recall that $L_{\infty\omega}=\bigcup_{\kappa}L_{\kappa^+\omega}$ and 
$L_{\infty\infty}=\bigcup_{\kappa}L_{\kappa^+\kappa^+}$.}
We first recall a basic construction in infinitary logic, based on a class of  important formulas due to D. Scott \cite{MR0200133}. For any ordinal $\alpha $ let the formula $\eta_\alpha(x)$ with one free variable $x$ and a binary relation symbol $<$ be defined, by transfinite recursion, as follows:
\begin{equation}\label{eta}
\eta_\alpha(x)\leftrightarrow\forall y(y<x\to\bigvee_{\beta<\alpha}\eta_\beta(y))
\wedge\bigwedge_{\beta<\alpha}\exists y(y<x\wedge\eta_\beta(y)).
\end{equation} Then for a linear order $(A,<)$ and $a\in A$ we have 
\begin{equation}\label{etaA}
(A,<)\models\eta_\alpha(a)\iff(\{b\in A:b<a\},<)\cong(\alpha,<).
\end{equation}
Let \begin{equation}\label{etap}
\eta'_\alpha\leftrightarrow\forall y\bigvee_{\beta<\alpha}\eta_\beta(y)
\wedge\bigwedge_{\beta<\alpha}\exists y\ \eta_\beta(y).
\end{equation}
For a linear order $(A,<)$ we have 
\begin{equation}\label{etapA}
(A,<)\models\eta'_\alpha\iff(A,<)\cong(\alpha,<).
\end{equation}
Now $\eta'_\alpha$ is in $L_{\kappa\omega}$ whenever $\alpha<\kappa$. It follows that by combining the sentences $\eta'_\alpha$, the logic $L_{\kappa\omega}$ can manifest totally arbitrary patterns of spectra in cardinalities below $\kappa$, roughly for the same reason that any finite set of finite numbers can be the spectrum of a first order sentence. More exactly, if $X$ is an arbitrary  set of cardinal numbers below $\kappa$, then $$\bigvee_{\lambda\in X}\eta'_\lambda\in L_{\kappa^+\omega}\mbox{ and }X=\sp(\bigvee_{\lambda\in X}\eta'_\lambda).$$ This shows that when dealing with extensions of $L_{\kappa\omega}$ it makes sense to focus on cardinals $\ge\kappa$. 

The basic facts about the $\LS(C,D)$ type properties of infinitary languages are the following:

\begin{theorem}[\cite{MR0539973}]\label{ls} Suppose $\kappa\le\lambda$ are cardinals and $\kappa$ is regular. 
\begin{enumerate}
\item $L_{\kappa^+\omega}$ satisfies $\LS([\lambda,\infty)\},\{\mu\})$ for all $\mu$ such that $\kappa\le\mu\le \lambda$.
\item $L_{\kappa^+\kappa^+}$ satisfies $\LS([\lambda,\infty),\{\mu\})$ for all $\mu$ such that $\kappa\le\mu\le \lambda$ and $\mu^\kappa=\mu$.
\item $L_{\kappa^+\omega}$ satisfies $\LS([\beth_{(2^\kappa)^+},\infty),[\mu,\infty))$ for all $\mu$.
\item $L_{\omega_1\omega_1}$ does not satisfy ``For all $\mu$ $\LS([\lambda,\infty),[\mu,\infty))$" for any $\lambda$ below the first inaccessible cardinal (assuming there are inaccessible cardinals).
\end{enumerate}
\end{theorem}

For a class size logic such as $L_{\infty\omega}$ or $L_{\infty\infty}$ the $\LS(C,D)$ properties hold  via their connection to $L_{\kappa\omega}$ or $L_{\kappa\lambda}$: If $\phi\in L_{\infty\infty}$, then $\phi\in L_{\kappa\lambda}$ for some (least) $\kappa$ and $\lambda$ and then Theorem~\ref{ls} applies.
%Suppose $\mm\models\phi$. Then there is $\mn\models\phi$ such that $|N|\le |M|\cdot\kappa$, where $\kappa$ is the least $\kappa$ such that $\phi\in H(\kappa^+)$. $L_{\infty\omega}$ satisfies this, $L_{\infty\infty}$ not. 

The above theorem shows that the spectra of $L_{\infty\infty}$ are much more complicated than the spectra of $L_{\infty\omega}$. For example, let $\theta$ be the sentence (\ref{nonabs}) of $L_{\omega_1\omega_1}$ which has a model of cardinality $\mu$ if and only if $\mu^\omega=\mu$.   Thus $\sp(\theta)$ (i.e. $\sp(\Mod(\theta))$) is full of ``holes" as it misses all cardinals, such as e.g. each $\aleph_{\alpha+\omega}$, that are $\omega$-cofinal. Whereas the spectra of sentences of $L_{\kappa\omega}$ cannot have such holes above $\kappa$.
(See Figure~\ref{pic1}.)

What does it reveal about logicality if the spectrum is full of  `holes'? {\em It suggests that we are very far from the situation in which we can claim we have the same logical operation independently of the domain}. In the case of $L_{\kappa^+\omega}$ the spectrum is (above $\kappa$) a homogeneous segment of cardinals, either bounded by 
$\beth_{(2^\kappa)^+}$ or unbounded, and we are closer to having  the same logical operation independently of the domain. Thus we judge a model class definable in $L_{\infty\omega}$ as having a greater degree of logicality than a model class definable in $L_{\infty\infty}$.
\begin{figure}
\begin{center}
\includegraphics[height=5cm]{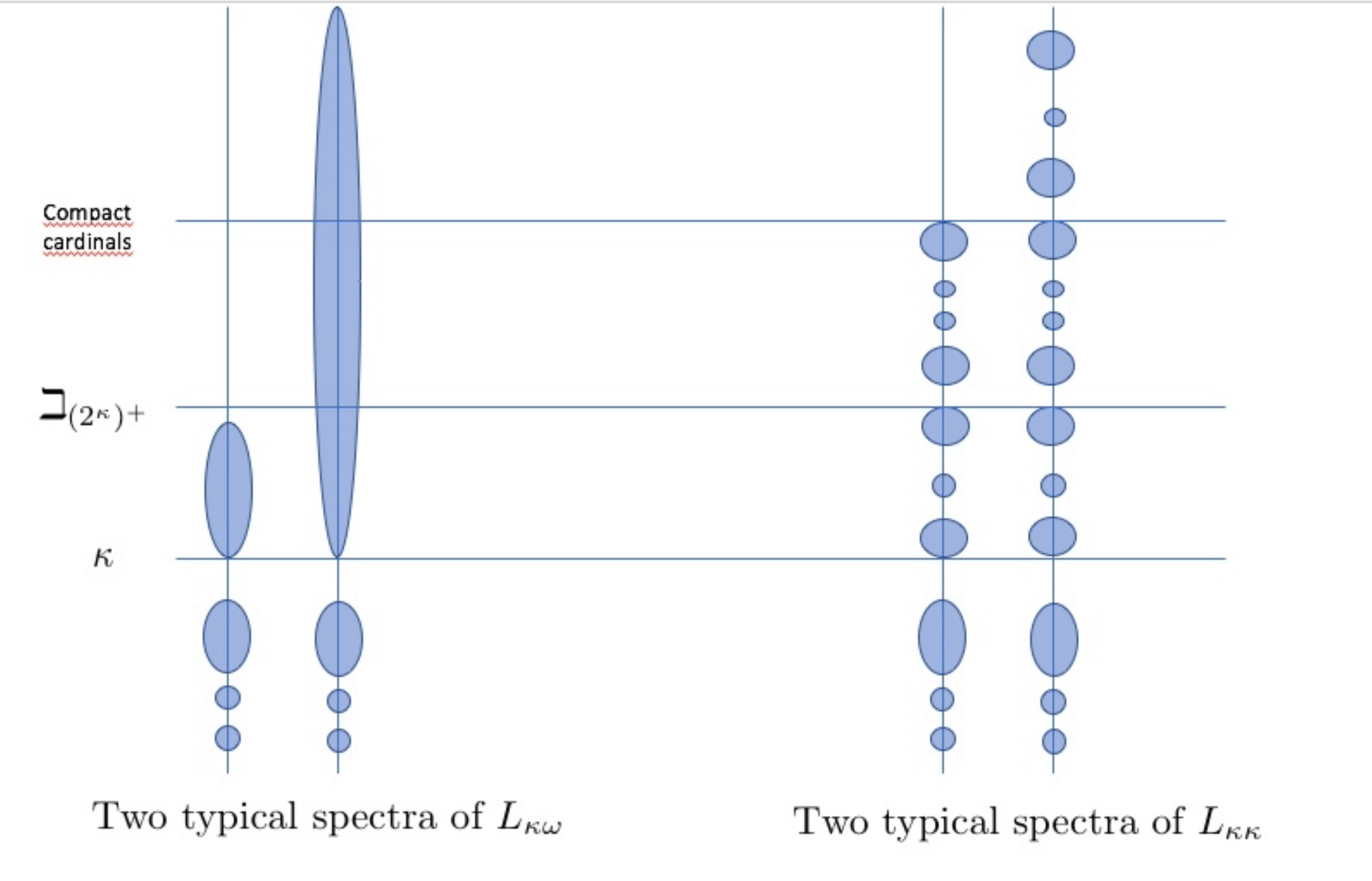}
\caption{\label{pic1}A difference between $L_{\infty\omega}$ and $L_{\infty\infty}$.}

\end{center}
\end{figure}

In Definition~\ref{lnum} above we defined the L\"owenheim the Hanf number of an arbitrary logic. These numebrs always exist if the class of formulas of the logic is a set. Thus they exist for $L_{\kappa\lambda}$ for any $\kappa$ and $\lambda$ but not for $L_{\infty\omega}$ and $L_{\infty\infty}$.

The Hanf-number  and L\"owenheim number of  $L^*$ can be easily defined in terms of spectra. Below a spectrum is called \emph{bounded} if it is a set (rather than a proper class).

\begin{equation}\label{hanf}
\begin{array}{lcl}
\ell(L^*)&=&\sup\{\min(\sp(\phi)):\phi\in L^*, \sp(\phi)\ne\emptyset \}\\
h(L^*)&=&\sup\{\sup(\sp(\phi)):\phi\in L^*
\mbox{ and $\sp(\phi)\ne\emptyset$ is bounded}\}.
\end{array}
\end{equation}

We saw that Sagi employs  L\"owenheim numbers as a measure of logicality: the smaller the L\"owenheim number the greater the degree of logicality. Thus model classes definable in $L(Q_\alpha)$ have a stronger degree of logicality than those definable in $L(Q_\beta)$, for $\beta>\alpha$. 

The logic $L(Q_\alpha)$ satisfies for trivial reasons the L\"owenheim-Skolem property $\LS([\omega,\aleph_\alpha),\{\mu\})$ for all $\mu<\aleph_\alpha$ and the rather strong L\"owenheim-Skolem property $\LS([\aleph_\alpha,\infty)),\{\mu\})$ for all $\mu\ge\aleph_\alpha$. Thus the spectra have no holes and $\ell(L(Q_\alpha))=\aleph_\alpha$. It should be noted that the class size logic $L(Q_\alpha)_{\alpha\in\On}$ is a sublogic of $L_{\infty\infty}$. Its spectra have a greater degree of regularity than those of $L_{\infty\infty}$, and thus in that respect it resembles more $L_{\infty\omega}$. In sum, these two logics, namely $L_{\infty\omega}$ and $L(Q_\alpha)_{\alpha\in\On}$, both manifest a greater degree of logicality than  $L_{\infty\infty}$, {\em if regularity in spectrum patterns is used as a criterion}.

\section{Is every model class definable in $L_{\infty\infty}$ irrespective of the cardinality?}\label{nope}

Why is it that we cannot replace condition (2) of Theorem~\ref{mcgee} with the apparently better condition
\begin{equation}\label{better}
\mbox{$\K$ is definable in $L_{\infty\infty}$?}
\end{equation}  This would solve the problems of cardinality and domain-relativity---the logical concept would have the same definition in terms of a formal language independently of the cardinality of the domain. 
However, there is a very simple reason why condition (2) of Theorem~\ref{mcgee} cannot be improved  to (\ref{better}): Not every model class is definable in $L_{\infty\infty}$. 
%We gave an example of such a class $W_2$ in (\ref{W2}). We now give  some others.

Recall the definition of $\K_{\mbox{\tiny lim}}$ above in Section~\ref{mcg}. It is, strictly speaking,  a generalized quantifier in the sense of \cite{MR89816} but a rather curious one. It is the existential quantifier in models of successor cardinality and the quantifier ``for no $x$" in models of limit cardinality.

\begin{proposition}\label{lim}
The model class $\K_{\mbox{\tiny lim}}$ is not definable in $L_{\infty\infty}$.
\end{proposition}

\begin{proof}
Similar to the proof of Proposition~\ref{har} below.
\end{proof}

A different kind of generalized quantifier  is the class $$W_2=\{(A,<): (A,<)\cong (\alpha+\alpha,\in)\mbox{ for an ordinal $\alpha$}\}.$$ This is potential isomorphism closed and not definable even in $L_{\infty\infty}$ \cite{MR290943}. Note that this model class is clearly second order definable.
The model class $W_2$ has a particularly nice spectrum, its infinite part being simply the class of all infinite cardinals. This model class is also set-theoretically absolute (see below section~\ref{abs}). In the light of the spectrum criterion, a model being an element of $W_2$ is a logical property of the model of a very high degree of logicality. Looking at a model, can we say that its being a well-order, let alone of a well-order of type $\alpha+\alpha$ for some $\alpha$, is a logical property? Undoubtedly some would say it is a mathematical rather than a logical property. But it is logical in the Tarski-Sher sense and it has a second order definition which is the same definition in each cardinality. However much it feels like a genuinely mathematical concept, it fulfils the Tarski-Sher criterion of logicality, escapes two of Feferman's criticisms, and fulfils Sagi's as well as Bonnay's criteria. So we may conclude that it is logical of a very high degree.

A third example of a model class  undefinable in $L_{\infty\infty}$ is particularly interesting, namely the H\"artig quantifier.
 
\subsection{The H\"artig Quantifier}\label{hartig}

The H\"artig quantifier is defined as follows:
$$\begin{array}{lcl}
Ixy\phi(x)\psi(y)&\iff&\mbox{ there are as many $x$ satisfying $\phi(x)$}\\
&&\mbox{as there are $y$ satisfying $\psi(y)$.}
\end{array}$$

\begin{proposition}[\cite{MR1136448}]\label{har}
The class of models $(M,A,B)$, where $A,B\subseteq M$ and $|A|=|B|$, is not definable in $L_{\infty\infty}$. Equivalently, the H\"artig quantifier 
is not definable in $L_{\infty\infty}$.
\end{proposition}

\begin{proof}
Two sentences of the logic $L_{\omega\omega}(I)$  are useful here: the sentence $\phi_{\lim}$ which has the class of limit cardinals as its spectrum, and the sentence  $\phi_{suc}$  which has the class of infinite successor cardinals as its spectrum. The sentence $\phi_{\lim}$ says that the universe is totally ordered by a linear order in which every element determines an initial segment which has smaller cardinality than the initial segment determined by some bigger element. The sentence $\phi_{suc}$ says that the universe is totally ordered by a linear order in which some element determines an initial segment which has the same cardinality as the initial segment determined by any bigger element, but smaller cardinality than the entire universe.  Suppose the H\"artig quantifier was definable in $L_{\kappa^+\kappa^+}$, where $\kappa$ is regular. Let $\lambda=2^\kappa$. 
If $\lambda$ is a limit cardinal, we can use $\phi_{\lim}$ to derive a contradiction as follows: Let $\mm$ be a model of $\phi_{\suc}$ of cardinality $\lambda^+$. By Theorem~\ref{ls} there is a model $\mn$ of cardinality $\lambda$ such that $\mn\models\phi_{\suc}$, a contradiction. The case that $\lambda$ is a successor cardinal is similar, but using $\phi_{\lim}$ instead of $\phi_{\suc}$.
\end{proof}

While the logic $L(I)$ with the H\"artig quantifier is not a sublogic of $L_{\infty\infty}$ it has some of the properties of the latter that we associate with being very far from first order logic (see Section~\ref{compl}). For example, it is highly non-absolute. It is also unbounded i.e. capable of defining well-ordering with additional predicates \cite{MR0244012}.

Note that the equicardinality concept, built around the existence of a bijection, lies at the very heart of the definition of logicality in the sense of Tarski! It is almost paradoxical then, that  equicardinality itself, or more precisely the logic built on the equicardinality quantifier, represents a particularly weak degree of logicality.

\subsection{Vop\v{e}nka's Principle}

While it is not true that every model class is definable in $L_{\infty\infty}$, there is a weaker result which depends on large cardinals. Tarski proved that a model class $\K$ is definable by a universal first order sentence if and only if $\K$ is closed under substructures and any model all of whose  finite substructures  are in $\K$, is itself in $\K$. The following result removes the assumption on finite substructures but lifts the result to an infinitary level. Recall that
\emph{Vop\v{e}nka's principle} is the axiom schema stating that if $\{A_\alpha : \alpha \in \On\}$ is a (definable) proper class of structures of the same vocabulary, there are $\alpha\ne \beta$ such that $A_\alpha$ can be  embedded into $A_\beta$. Despite its formulation, which suggests no connection to large cardinals, this principle is actually a large cardinal axiom schema. It implies the existence of extendible (and hence supercompact) cardinals \cite{MR295904}, and its consistency follows from the existence of an almost huge cardinal (see e.g. \cite[p. 338]{MR2731169}). Magidor has proved the following characterisation of Vop\v{e}nka's Principle:\footnote{We present the proof with Professor Magidor's kind permission.}

\begin{theorem}[\cite{magidorCIRM}]
The following are equivalent:
\begin{enumerate}
\item Every model class which is closed under substructures is $L_{\infty\infty}$-definable by a universal sentence.
\item Vop\v{e}nka's Principle.
\end{enumerate}
\end{theorem}

\begin{proof} (2) implies (1): Suppose $\K$ is a model class which is closed under substructures. By Vop\v{e}nka's Principle there is a $\K$-supercompact cardinal $\kappa$ (see \cite{MR482431}). 
For any $\mm$ of cardinality $\le\kappa$, let $\theta_\mm\in L_{\kappa^+\kappa^+}$ be an existential sentence such that $\theta_\mm$ is true exactly in models which have an isomorphic copy of $\mm$ as a substructure. We show that for all $\mn$, $\mn\notin \K$ if and only if $\mn\models\psi$, where $\psi$ is $\bigvee\{\theta_\mm: M\subseteq\kappa, \mm\notin \K\}$. Suppose first $\mn\notin \K$. By the $\K$-supercompactness of $\kappa$ there is an $L_{\omega\omega}(Q_{\K})$-elementary substructure $\mm$ of $\mn$ of cardinality $\le\kappa$. Now $\mn\notin \K$ implies $\mm\notin \K$. Trivially, $\mn\models\theta_\mm$. Hence $\mn\models\psi$. Conversely, suppose $\mn\models\theta_\mm$ for some $\mm\notin \K$ such that $M \subseteq\kappa$. W.l.o.g. $\mm\subseteq\mn$. By  closure under substructures, $\mn\notin \K$.

(1) implies (2): Suppose Vop\v{e}nka's Principles fails as witnessed by the class $\C=\{\mm_\alpha : \alpha\in\On\}$. Let $\K$ be the class of structures $\mn$ of the same vocabulary as each $\mm_\alpha$ such that $\mn$ is isomorphic to an  substructure of  $\mm_\alpha$ for arbitrarily large $\alpha$. By (1) we may choose $\kappa$ such that $\K$ is definable in $L_{\kappa^+\kappa^+}$. By our choice of  $\C$, every $\mm_\alpha$ satisfies $\mm_\alpha\notin\K$. By 
Theorem~\ref{ls} there is for each $\alpha$ an $L_{\kappa^+\kappa^+}$-elementary substructure $\mn_\alpha$ of $\mm_\alpha$ of cardinality $\le 2^\kappa$ such that $\mn_\alpha\notin\K$. There are only a set of non-isomorphic structures of cardinality  $\le 2^\kappa$. Hence $\mn_\alpha\cong\mn_\beta$ for a proper class of $\alpha$ and $\beta$. Let $\alpha$ be one of those. Then $\mn_\alpha\in\K$, a contradiction.
\end{proof}

The moral of the story here is that we can obtain the McGee style characterization of logicality independently of the cardinality of the model if we restrict to model classes closed under substructures, but we have to make a strong set theoretical assumption, namely Vop\v{e}nka's Principle. This is not to suggest that we propose closure under substructures as a criterion for logicality.

\section{Defining an arbitrary model class in an absolute logic, but still cardinal dependently}\label{abs}

We shall now offer a refinement of McGee's result (Theorem~\ref{mcgee}), by replacing $L_{\infty\infty}$ by a substantially weaker infinitary logic\footnote{McGee alludes to such a possibility on p. 574 of \cite{MR1426482}.}. Let us first note that we cannot replace $L_{\infty\infty}$ by $L_{\infty\omega}$ because the quantifier $Q_1$ is not CD-definable in the latter.\footnote{Suppose $\mm$ is a model of cardinality $\aleph_1$ with one unary predicate $P$ and $P$ is countable. Suppose $\mn$ is another   model of cardinality $\aleph_1$ for the same vocabulary but now both $P$ and its complement are uncountable. Clearly $\mm$ and $\mn$ are partially (potentially) isomorphic and hence satisfy the same $L_{\infty\omega}$-sentences. But $Q_1xP(x)$ is true in one but not in the other.} For another example, the well-ordering quantifier
$$\mm\models Wxy\phi(x,y,\vec{a})\iff\{(c,d):\mm\models\phi(c,d,\vec{a})\}\mbox{ well-orders $M$}$$ is definable in $L_{\omega_1\omega_1}$ (see (\ref{wf})) but not  in $L_{\infty\omega}$ \cite{MR200131}.

An important quality of $L_{\infty\omega}$ is its absoluteness. Let us recall:

\begin{definition}[\cite{MR0337483}]\label{abso}
A logic $(\Sigma,T)$ is called an \emph{absolute logic} if the predicate $\Sigma(\phi)$ is $\Sigma_1$ in $\phi$, and the predicate $T(\phi,\mm)$ is $\Delta_1$ in $\phi$ and $\mm$. A logic is absolute with respect to a set theory $S$ if the predicates $\Sigma$ and $T$ are $\Sigma_1^S$ and $\Delta_1^S$, respectively.
\end{definition}
%We define the $\Delta$ extension of a logic below. 

For example, $L_{\infty\omega}$ and $L(W)$ are absolute logics but $L_{\infty\infty}$ is not.

We defined the $\Delta$-operation above in Definition~\ref{sigma11}.
As argued for in \cite{MR457146}, $\Delta(L^*)$ can be considered a logic in itself. There is a many-sorted version and a single-sorted version of the $\Delta$-extension. In the many-sorted version the model class $\K$ may have a many-sorted vocabulary. If $L^*$ is the logic $L_{\infty\omega}$, or $L_{\kappa^+\omega}$, for regular $\kappa$, there is no difference between the two versions i.e. they coincide  \cite{MR606603}. Therefore we continue with the single-sorted version.

Some of the nice properties of the $\Delta$-operation are:
\begin{itemize}
\item Exactly the same classes of cardinals are spectra of $\Delta(L^*)$ as are of $L^*$.
\item $\Delta(L^*)$ satisfies $LS(C,D)$ iff $L^*$ does.
\item $\Delta(L^*)$ satisfies the Compactness Theorem iff $L^*$ does.
\item $\Delta(L^*)$ satisfies the (abstract) Completeness Theorem\footnote{I.e. the set of G\"odel numbers of valid sentences is r.e. This is only meaningful for $L^*$ where formulas are finite objects so that G\"odel numbering makes sense.} iff $L^*$ does.
\item If $L^*$ satisfies the Craig Interpolation Theorem, then $\Delta(L^*)$ and $L^*$ are equivalent logics.
\item $\ell(\Delta(L^*))=\ell(L^*)$.
\item $h(\Delta(L^*))=h(L^*)\footnote{This may fail in the many-sorted version \cite{MR690669}.}.$
\end{itemize}

If $L^*$ is an absolute logic in the sense of \cite{MR0337483}, then $\Delta(L^*)$ is still absolute in the weaker sense that every definable model class is $\Delta_1$-definable in set theory (with the defining sentence as a parameter). This is a direct consequence of the definition of $\Delta(L^*)$. However, $\Delta(L^*)$ has a complication which prevents us from concluding that the $\Delta$-operation preserves absoluteness. Namely, there may be an absolute logic $L^*$ and model classes $\K_0$ and $\K_1$ such that both $\K_0$ and $\K_1$ are in $\Sigma(L^*)$, $\K_0\cap \K_1= \emptyset$ but the proposition  that every model (of the right type) is in $\K_0$ or $\K_1$ is not absolute\footnote{For example, we may take $\K_0$ to be the class of trees of height and size $\omega_1$ with an uncountable branch, and $\K_1$ to be the class of trees of height and size $\omega_1$ with a strict order preserving mapping into the rational numbers. It is a consequence of Martin's Axiom that every such tree is  in $\K_0\cup \K_1$. However, if $V=L$, then there are Souslin trees which are not in $\K_0\cup \K_1$.}. Thus we may not have a $\Sigma_1$-definition for the set of sentences of $\Delta(L^*)$ even if we have for $L^*$. Despite this shortcoming of the $\Delta$-operation, we do have a $\Sigma_1$-definition for the set of sentences of $\Sigma(L^*)$ when $L^*$ is absolute because with $\Sigma(L^*)$ the above problem ($\K_0$ and $\K_1$) does not arise. 

\begin{theorem}[\cite{Kennedy2021-KENGTA-2}]\label{impr}
If $\K$ is a class of models of the same  vocabulary, then the following conditions are equivalent:
\begin{enumerate}
\item $\K$ is closed under isomorphisms.
\item $\K$ is CD-definable in $\Delta(L_{\infty\omega})$.
\end{enumerate}
\end{theorem}

\begin{proof}The proof is similar to the proof of Theorem~\ref{mcgee} in \cite{MR1426482}.
Recall the formulas $\eta'_\alpha$ and $\eta'_\alpha$ defined in (\ref{eta}) and (\ref{etap}), with the properties (\ref{etaA}) and (\ref{etapA}).
Suppose $\mak$ is a model of the vocabulary $L$ and $|A|=\lambda$. Let $f_A:\lambda\to A$ be a bijection and $a<_Ab\iff f_A^{-1}(a)<f_A^{-1}(b)$. For $\alpha,\beta<\lambda$ let
$$\rho_{\alpha,\beta}(x,y)=\left\{\begin{array}{ll}
R(x,y)&\mbox{ if $\mak\models R(f_A(\alpha),f_B(\beta))$}\\
\neg R(x,y)&\mbox{ if $\mak\not\models R(f_A(\alpha),f_B(\beta))$}
\end{array}\right.$$
Let
$$\Phi_\mak\leftrightarrow\forall x\forall y\bigwedge_{\alpha,\beta<\lambda}((\eta_\alpha(x)\wedge\eta_\beta(y))\to\rho_{\alpha,\beta}(x,y)).$$Now $(\mak,<_A)\models\eta'_\lambda$ and $(\mak,<_A)\models\eta_\alpha(a)\iff a=f_A(\alpha)$. Hence
$(\mak,<)\models\eta'_\lambda\wedge\Phi_\mak.$ On the other hand, 
$(\mak',<')\models\eta'_\lambda\wedge\Phi_\mak$ implies $\mak'\cong\mak,$
for if $(\mak',<')\models\eta'_\lambda\wedge\Phi_\mak$ and $g:(A',<')\cong(A,<_A)$, then  
$g:\mak'\cong\mak$. Let $$\Theta_\K\leftrightarrow\bigvee\{\Phi_\mak\ : \ \mak\in\K, A=\lambda\}.$$ Now by the above, 
$$\begin{array}{lcl}
\mak\in\K_\lambda&\iff&(\mak,<)\models\eta'_\lambda\wedge\Theta_\K\mbox{ for some $<$ }\\
&\iff&(\mak,<)\models\eta'_\lambda\to\Theta_\K\mbox{ for all $<$}.
\end{array}$$ Since $\eta'_\lambda,\Theta_\K\in L_{(2^\lambda)^+\omega}$, we are done.
\end{proof}

Why is this  an improvement? $\Delta(L_{\infty\omega})$ is a  sublogic of $L_{\infty\infty}$, a consequence of the result of Malitz \cite{MR290943} to the effect that for regular $\kappa$ a valid implication in $L_{\kappa\omega}$ can be interpolated in $L_{2^{(<\kappa)^+\kappa}}$. It is a proper sublogic because well-ordering is definable in the latter but not in the former. Secondly, the logic $L_{\infty\omega}$ is an absolute logic \cite{MR0337483} and  $\Delta(L_{\infty\omega})$ inherits much of the absoluteness of  $L_{\infty\omega}$ (see above) while $L_{\infty\infty}$ is badly non-absolute. 

The improvement that Theorem~\ref{impr} represents over Theorem~\ref{mcgee} is also based on the fact that  $L_{\infty\omega}$, and thereby also $\Delta(L_{\infty\omega})$, is a much ``tamer" logic than $L_{\infty\infty}$. In particular:

\begin{itemize}
\item $\Delta(L_{\infty\omega})$, even $\Sigma(L_{\infty\omega})$, has a strong L\"owenheim-Skolem theorem. This is in contrast to  $L_{\infty\infty}$, which does not have a L\"owenheim-Skolem Theorem in the same strong form (see Theorem~\ref{ls}).
%\item In $L_{\infty\infty}$ there are (e.g.) sentences which have  models in exactly the cardinalities $\mu$ such that $\mu^\omega=\mu$. Such a $\mu$ cannot have cofinality $\omega$. Thus $L_{\infty\infty}$ does not have a L\"owenheim-Skolem Theorem in the same strong form as $\Delta(L_{\infty\omega})$.
\item The class of well-orderings is not definable in $\Delta(L_{\infty\omega})$, while it is definable in $L_{\omega_1\omega_1}$ (see (\ref{wf}) below). This is a kind of threshold difference. In general, logics in which the concept of well-order is not definable are much better behaved (in many different respects) than those in which it is. If well-order is definable, the logic can talk about well-founded models of set theory, and thereby about transitive models of set theory. In this way the logic can break through the object theory/metatheory barrier and have access to the background set theory. For a concrete example, consider a sentence $\phi$ which is the conjunction of a sufficiently large finite part $T_0$ of the ZFC axioms, the first order set-theoretical statement $\phi_0$ that there are no inaccessible cardinals,  and the $L_{\omega_1\omega_1}$-sentence
\begin{equation}\label{wf}
\displaystyle{\forall x_0\forall x_1\ldots\bigvee_{n<\omega}}\neg x_{n+1}\in x_n.
\end{equation}
Suppose $\kappa$ is the smallest inaccessible cardinal. Then $\phi$ has a model of cardinality $\kappa$, namely $V_\kappa$, but none of bigger cardinality. For suppose $M$ is a model of $\phi$ of cardinality $>\kappa$. By Mostowski's Collapsing Lemma \cite{MR35721} we may assume $M$ is a transitive model of $T_0$. Since the cardinality of $M$ is bigger than $\kappa$, the ordinal $\kappa$ is in $M$ and is therefore inaccessible in $M$ by the downward persistency of inaccessibility. But this contradicts the fact that $M$ is a model of $\phi_0$. This argument, due to Silver \cite{MR409188}, demonstrates the power of (\ref{wf}) to penetrate the object theory/metatheory barrier, which sentences of $\Delta(L_{\infty\omega})$ are unable to do \cite{MR200131}.
 
\item The Hanf number of $\Delta(L_{\kappa^+\omega})$ is only moderately large, namely $<\beth_{(2^\kappa)^+}$, while the Hanf number of $L_{\omega_1\omega_1}$ is bigger than the first weakly inaccessible cardinal\footnote{See argument  in the previous bullet.} (and consistently bigger than the first measurable cardinal). The smallness of the Hanf number is another indication of regularity of patterns in spectra (see (\ref{hanf})).
\end{itemize}

A weakness of $\Delta(L_{\infty\omega})$ as compared to $L_{\infty\infty}$ is that the former does not have as explicit a syntax as the latter, although we can overcome this by replacing $\Delta(L_{\infty\omega})$ by $\Sigma(L_{\infty\omega})$. A strength is that its spectra are as regular as those of $L_{\infty\omega}$, it is equally completely axiomatizable as $L_{\infty\omega}$, and it is absolute in the sense that its definable model classes are absolute in set theory. Thus it avoids Feferman's second and third criticisms.

\section{Sort Logic}\label{sort}

We can ask, is there a logic in which every model class whatsoever is definable, irrespective of the cardinality of the domain? That is, not just cardinal dependently? We know already  $L_{\infty\infty}$ is not that kind of a logic (Section~\ref{nope}). McGee points out that we could take a class size disjunction of $L_{\infty\infty}$ sentences and obtain a single `sentence' which works in all cardinalities. As he points out, if we accept class size formulas we should also accept class size logical operations and then we are back in the starting point: to account for class size logical operations we need classes of classes and we start climbing up the type hierarchy beyond first order set theory. If we stick to set size logical operations we can operate within first order set theory.

We could, of course,  take every model class as a generalized quantifier\footnote{In the sense of section~\ref{mcg} and of \cite{MR0244012}.} but there is a more canonical logic for this task.

In \cite{MR3205075} a logic $L^s$ called \emph{sort logic} was introduced. It is a kind of many-sorted version of second order logic $L^2$, which allows quantification not only over subsets and relations on the domain of the model but also over subsets and relations on new domains {\em outside} the current one. Such quantification happens, for example, when we ask of a given group, whether it is the multiplicative group of a field? We have to ``guess" the addition of the field but also the neutral element $0$ and those are both outside the multiplicative group.

 More exactly, sort logic arises from $L^2$ by repeated applications of operation $L^*\mapsto\Sigma(L^*)$ (see Definition~\ref{sigma11}) and negation.   Let us define $\Delta_0=L^2$. If $\Delta_n$ has been defined, let $\Delta_{n+1}$ consist of model classes $\K$ such that both $\K$ and the complement of $\K$ are $\Sigma(\Delta_n(L^2))$-definable. We get an increasing hierarchy of logics
$$\Delta_1\le\Delta_2\le \Delta_3\le \ldots$$
the union of which is sort logic. A fine point is that every level $\Delta_n$ of sort logic is definable in set theory but there is no single formula that defines the entire sort logic.

\begin{theorem}[\cite{MR3205075}] If $\K$ is a class of models of the same finite vocabulary, definable in set theory without parameters,\footnote{We may allow parameters if we extend sort logic so that it includes $L_{\infty\omega}$.} then the following conditions are equivalent:
\begin{enumerate}
\item $\K$ is closed under isomorphisms.
\item $\K$ is definable in sort logic.
\end{enumerate}
\end{theorem}

The above characterization of isomorphism closure is not as elegant as McGee's Theorem~\ref{mcgee} simply because the definition of sort logic is  more complicated than that of $L_{\infty\infty}$. However, this theorem has the remarkable advantage over Theorem~\ref{mcgee}  that the definition of the model class in sort logic holds for models of \emph{all} cardinalities. As McGee points out, if he wanted to obtain the same level of generality with his method, he would have to take a disjunction of a proper class of  $L_{\infty\infty}$-formulas. Sort logic may be complicated but at least its formulas are sets and not proper classes. In McGee's case we may need a formula which is not even a set. In the case of sort logic each formula is a set, but the whole logic is not definable, only each level $\Delta_n$ separately is.

%Model theoretic properties of sort logic?
 
\section{Is having a Completeness Theorem a marker of logicality?}\label{compl}

We saw that first order logic and $\mathcal{L}(Q_0)$ are maximally logical under Sagi's criterion, grading the logicality of these logics by their L\"owenheim numbers. We also saw that the degree of logicality of  the logics $\mathcal{L}(Q_{\alpha})$ decreases as $\alpha$ increases: if $\alpha \leq \beta$, then $Q_{\beta}$ is less logical than $Q_{\alpha}$.

As for other kinds of quantifiers, the quantifiers  ``more" or Rescher quantifier 
$$\begin{array}{lcl}
Jxy\phi(x)\psi(y)&\iff&\mbox{ there are at least as many $x$ satisfying $\phi(x)$}\\
&&\mbox{as there are $y$ satisfying $\psi(y)$}
\end{array}$$and the equicardinality or H\"artig quantifier (see Section~\ref{hartig}) both have the same (very high) L\"owen\-heim number, and are thus less logical than $Q_{\alpha}$ at least for $\alpha$ below the first $\alpha$ such that $\alpha=\aleph_\alpha$.\footnote{If $\ell_I$ is the L\"owenheim number of the H\"artig quantifier, then $\ell_I$ is always bigger than the first fixed point of the $\aleph$-hierarchy \cite{MR304130}.
If $V=L$, then $\ell_I$ is bigger than the first inaccessible (if any exist) \cite{MR585518}.
If $V=L^\mu$, then $\ell_I$ is bigger than the first measurable cardinal \cite{MR510714}.
 If Con(``there is a super compact cardinal"), then Con($\ell_I<$ the first weakly inaccessible) \cite{MR2833152}.
 If Con($\ZFC$), then Con($\ell_I<2^\omega$) \cite{MR585518}.} We suggested earlier that the  H\"artig quantifier is a singular case, in seeming to express the core concept of isomorphism invariance, which is defined in terms of the concept of equicardinality. Under our criteria, the H\"artig quantifier is classified as (only) weakly logical even so, not only due to its high L\"owenheim number but also because it lacks very seriously a complete axiomatization \cite{MR585518}.

We remarked earlier that Sagi's demarcation of  logicality for a fixed class of logical constants, which is based on L\"owenheim numbers,  partitions  logical space  differently than ours. As we observed, logics that have a completeness theorem do not seem to be tied to the $\aleph$-hierarchy in any obvious way.
%, in that  some of these logics are axiomatisable and some are not. 
%The connection is the following: the more logical the terms that you fix in a system, the less the resulting consequence relation is metaphysically involved.
For example, Keisler's axioms 
\begin{equation}\label{ka}
\left\{\begin{array}{ll}
\mbox{Axiom $0$}&\mbox{Axiom schemes for $\mathcal{L}_{\omega\omega}$}\\
\mbox{Axiom $1$}&\neg Qx(x=y\vee x=z)\\
\mbox{Axiom $2$}&\forall x(\phi\to\psi)\to(Qx\phi\to Qx\psi)\\
\mbox{Axiom $3$}&Qx\phi(x)\to Qy\phi(y)\\
\mbox{Axiom $4$}&Qy\exists x\phi\to(\exists xQy\phi\vee Qx\exists y\phi)
\end{array}\right.
\end{equation}
are complete for $\mathcal{L}(Q_1)$, and if $GCH$ is assumed, then they are complete for all $\aleph_{\alpha+1}$ for which $\aleph_{\alpha}$ is regular (hence for $\aleph_n$ for all $n>0$).  
Whereas while $\mathcal{L}(Q_0)$ satisfies the Keisler Axioms, these (or any other recursive set of axioms) are not a complete axiomatisation of   $\mathcal{L}(Q_0)$.\footnote{See \cite{MR89816}.
} 

The suggestion here is that having a complete axiomatisation should be a marker of logicality on the simple ground that logics of this kind resemble  first order logic in one of its essential, if not most essential property, namely completeness. Except for $Q_0$, one would then classify ``there are very (i.e. uncountably) many"  as logical, as dictated by Keisler's axioms; in particular, the quantifier $Q_1$ would be graded as having a higher degree of logicality  than $Q_0$, as $\mathcal{L}(Q_1)$ is complete with respect to the Keisler Axioms
whereas  $\mathcal{L}(Q_0)$ is not complete in this respect.\footnote{However, $\mathcal{L}(Q_0)$ is the only one among the logics $\mathcal{L}(Q_\alpha)$ which satisfies the same logical consequences as $\mathcal{L}(Q_0)$. See \cite{dag}.} In short, our criterion of logicality, which turns on the degree to which a logic resembles first order logic in its model theoretic properties, comes apart from Sagi's already at the level of the logics $\mathcal{L}(Q_0)$ and $\mathcal{L}(Q_1)$.

A possible objection  might come from the fact that  Keisler's axioms are satisfied by many \emph{different} quantifiers, in fact by every $Q_\alpha$, where $\aleph_\alpha$ is regular. On the other hand, we may consider $Q_1$ logical admitting that from the point of view of logicality it cannot be separated from some other similar quantifiers. What is logical, according to this view, is not so much the cardinality $\aleph_1$, as, simply, uncountable cardinality---or ``very many"---in general, while the failure of $Q_0$ to permit a recursive axiomatization is an indication that there is something mathematical, as against logical, in $Q_0$. The permutation invariance characterization of logicality does not differentiate between $Q_1$ and $Q_0$, as the Completeness Theorem criterion does. By the invariance criterion every $Q_\alpha$ is logical, which may seem unintuitive. By further applying the Completeness Theorem criterion we can make finer distinctions and see a difference: $Q_1$ is ``more logical" than $Q_0$. 

For singular strong limit $\aleph_{\alpha}$ Keisler \cite{MR0269491} proved a Completeness Theorem, but with different axioms (no simple set of axioms is currently known). Thus there are (at least) two different ``logical" concepts of ``very many," one for $\aleph_{\alpha}$ the successor of regular and one for the singular case. (Successor of singular is open.)  The two concepts have different logical content: Keisler's Axioms for $Q_{\alpha}$ are valid for all regular $\aleph_\alpha$, and in some cardinalities (successor of regular) they have a Completeness theorem, modulo the $GCH$, as was noted above. In some other cardinalities (singular) different axioms have to be employed in order to get a completeness theorem. In the base case $\alpha=0$ no effectively given additional axioms can be added to give a completeness theorem.  

According to Sagi,  ``We should distinguish between a logic $\mathcal{L}(Q)$ used to measure the logicality of $Q$ and the logic we ultimately use for validity and logical consequence."\footnote{ibid, p. 22.} But if metaphysical involvement is inversely related to logicality, then having a completeness theorem should possibly be considered as a marker of logicality also under the Sagi criterion. This is because what the completeness theorem precisely does is to enable the conversion of semantic content, which is metaphysically involved (from the Sagi point of view, presumably), into syntactic content which, presumably, is not. 
In this connection, the following caveat is important: As Carnap observed \cite{MR0007892} the rules do not fix the interpretation of the logical constants, they have non-standard interpretations. In their \cite{MR3531783} Bonnay and Westerst\aa hl present a way to sidestep the problem:

\begin{quote}
Our take on Carnap's Problem is that it is made artificially difficult by considering all possible interpretations, no matter how bizarre. As speakers, we know that our language is going to be compositional, that it will have some true and some false sentences, and that its logical constituents will be topic-neutral. Therefore attention may be restricted to interpretations which satisfy these principles. Following Church's advice, this amounts to explicitly factoring out the role of semantic principles and the role of inference rules in fixing the interpretation of logical constants, rather than covertly using semantic notions to make sense of extended inference rules. This strategy proves successful both for propositional connectives and for quantifiers.
\end{quote} 

There is also the issue of expressive power, with respect to which the logics built on  the quantifiers  ``most", ``more" and the H\"artig quantifier differ. In certain models with an equivalence relation the H\"artig quantifier can be seen to be eliminable while the Rescher quantifier is not \cite{MR617189}. Thus the Rescher quantifier is strictly stronger than the H\"artig quantifier from the point of view of expressive power. The point here is that expressive power should be inversely related to logicality. 

Bonnay also ties logicality to syntax via a completeness theorem, as is spelled out in some detail in  \cite{bonnayrecent}.\footnote{See also the recent \cite{bon}.} Here Bonnay proposes a modification of the program of Carnap's {\em Aufbau} \cite{MR1661039},   calling for logical expressions to be defined syntactically. It is a feature of his treatment that  the absoluteness of a logic plays a crucial role in guaranteeing the {\em robustness} of the syntactic definition. Thus $L(Q_0)$ is an absolute logic because it has a recursive syntax, just like first order logic, and its semantics is absolute in transitive models of (even weak) set theory.\footnote{This is essentially because finiteness is absolute, i.e.  $\Delta_1$-definable. See  Barwise ``Admissible sets and Structures" \cite{MR0424560}, p. 38.}
On the other hand, $L(Q_1)$ is not absolute, even though it has a recursive syntax.\footnote{This is due to the well-known fact that countability is not absolute: a set can be uncountable in a model of set theory and countable in a transitive extension. On the relevance of absoluteness in this context see also Bonnay's \cite{MR2395046}.}

Burgess's  \cite{MR497917} forges a link between absoluteness and the idea of having a proof procedure. He  exhibits a quasi-constructive complete proof procedure involving rules with $\aleph_1$ premisses for the hereditary countable part of any absolute logic. 

In sum, if one classifies the first order existential quantifier  ``there is at least one," as  inherently logical, and if as such this quantifier is thought of as having  minimal or no semantic content, then is there a principled way to determine when and how higher quantification acquires semantic content, e.g. at what level in the cumulative hierarchy? Sagi's  answer is that logicality diminishes the higher up we are in cumulative hierarchy; while we suggest  that logicality kicks in arbitrarily high up, e.g. for all $\aleph_{\alpha+1}$ for which $\aleph_{\alpha}$ is regular, with the $\aleph_0$ case an anomaly, again because of completeness. The intuition here is that logics which have a completeness theorem are close to first order logic in this special sense. This is because completeness enables the conversion of semantic consequence into syntactic consequence via the two Keisler axiomatisations---albeit conditioned on the continuum hypothesis in certain important cases.\footnote{See  \cite{Kennedy2021-KENGTA-2} for a similar treatment of the relation between logicality and completeness.}

\section{Conclusions}

%It has been generally acknowledged that the Tarski-Sher criterion for logicality is necessary but not sufficient. 
McGee's Theorem translates the Tarski-Sher criterion into definability in a logic, obtaining cardinal dependent definability in $L_{\infty\infty}$. The cardinal dependency aspect of the theorem leads to criticism of logicality across domains (Feferman's third critique), to criticism of entanglement with mathematics (Feferman's first critique) and of non-absoluteness (Feferman's second critique). Maintaining cardinal dependency we lowered $L_{\infty\infty}$ to $\Delta(L_{\infty\omega})$ which is more palatable from the point of view of model theoretic properties. A stronger degree of logicality is obtained by abandoning cardinal dependency and investigating logics in which a candidate for logicality is definable. Since every class of models which is closed under isomorphisms is definable in some logic, this seems reasonable. Following and expanding on Sagi's suggestion to delineate degrees of logicality  according to their L\"owenheim numbers, we delineate logicality according to their L\"owenheim numbers but also according to a wider spectrum of model theoretic properties of logics, such as  Completeness Theorems and absoluteness properties, together with L\"owenheim-Skolem properties, all considered from a logicality point of view.

%\bibliographystyle{plain}
%\bibliography{ffbib3,ffbib2}

\end{document}